\documentclass[reqno,11pt]{amsart}

\usepackage{fullpage}

\usepackage{amssymb}
\usepackage{dsfont}

\usepackage[utf8]{inputenc} % Required for inputting international characters
\usepackage[T1]{fontenc} % Output font encoding for international characters

\usepackage{hyperref}

\usepackage[nobysame,alphabetic,initials,msc-links]{amsrefs}

\DefineSimpleKey{bib}{how}

\renewcommand{\eprint}[1]{#1}
\BibSpec{misc}{%
  +{}{\PrintAuthors}  {author}
  +{,}{ \textit}      {title}
  +{,}{ }             {how}
  +{}{ \parenthesize} {date}
  +{,} { available at \eprint}        {eprint}
  +{,}{ available at \url}{url}
  +{,}{ }             {note}
  +{.}{}              {transition}
}

\numberwithin{equation}{section}

\theoremstyle{plain}%default
\newtheorem{thm}{Theorem}[section]
\newtheorem{prop}[thm]{Proposition}
\newtheorem{lemma}[thm]{Lemma}
\newtheorem{cor}[thm]{Corollary}
\theoremstyle{definition}

\newtheorem{defn}[thm]{Definition}
\theoremstyle{remark}

\newtheorem{remark}[thm]{Remark}
\newtheorem{example}[thm]{Example}

\newcommand\bp{\begin{proof}}
\newcommand\ep{\end{proof}}

\newcommand{\un}{\mathds{1}}

\newcommand\C{\mathbb{C}}
\newcommand\Z{\mathbb{Z}}

\newcommand{\A}{{\mathcal{A}}}
\newcommand{\G}{{\mathcal{G}}}
\newcommand{\TT}{{\mathcal{T}}}
\newcommand{\V}{{\mathcal{V}}}

\newcommand{\CcG}{{C_{c}(\mathcal{G})}}

\newcommand{\Gu}{{\mathcal{G}^{(0)}}}
\newcommand{\Gxx}{\mathcal{G}^x_x}

\newcommand\Ind{\operatorname{Ind}}
\newcommand\Mat{\operatorname{Mat}}
\newcommand\supp{\operatorname{supp}}

\newcommand\eps{\varepsilon}

\begin{document}

\title{(Non)exotic completions of the group algebras of isotropy groups}

\date{December 21, 2020; minor changes April 13, 2021}

\author{Johannes Christensen$^{1}$}
\thanks{$^{1}$KU Leuven, Department of Mathematics (Belgium).  E-mail:  johannes.christensen@kuleuven.be}

\author{Sergey Neshveyev$^{2}$}
%\address{Universitetet i Oslo}
%\email{sergeyn@math.uio.no}
\thanks{$^{2}$University of Oslo, Mathematics institute (Norway).  E-mail:  sergeyn@math.uio.no}

\thanks{J.C. is supported by a DFF-International Postdoctoral Grant and partially by the DFF project no. 7014-00145B. S.N. is partially supported by the NFR funded project 300837 ``Quantum Symmetry''.}

\begin{abstract}
Motivated by the problem of characterizing KMS states on the reduced C$^*$-algebras of \'etale groupoids, we show that the reduced norm on these algebras induces a C$^*$-norm on the group algebras of the isotropy groups. This C$^*$-norm coincides with the reduced norm for the transformation groupoids, but, as follows from examples of Higson--Lafforgue--Skandalis, it can be exotic already for groupoids of germs associated with group actions. We show that the norm is still the reduced one for some classes of graded groupoids, in particular, for the groupoids associated with partial actions of groups and the semidirect products of exact groups and groupoids with amenable isotropy groups.
\end{abstract}

\maketitle

\section*{Introduction}

Over the last fifty years there has been a significant interest in describing KMS states for different C$^{*}$-dynamical systems. By now there are several strategies how to approach this problem in concrete examples. One such strategy, which has proven to be extremely useful, is to represent a given system as a C$^*$-algebra of a locally compact \'etale groupoid $\G$ with the time evolution defined by a continuous real-valued $1$-cocycle $c$ on the groupoid. By a result of Renault~\cite{Rbook}, if the groupoid is Hausdorff and principal, then all KMS$_\beta$ states on $C^*(\G)$ arise by integration with respect to quasi-invariant probability measures $\mu$ on $\Gu$ with Radon-Nikodym cocycle $e^{-\beta c}$, and they all factor through the reduced groupoid C$^{*}$-algebra $C^*_r(\G)$. The situation for non-principal groupoids is more complicated.  As was shown by the second author~\cite{N}, in this case the KMS states are classified by pairs $(\mu,\{\varphi_x\}_x)$ consisting of a quasi-invariant measure $\mu$ and a measurable field of tracial states $\varphi_x$ on the C$^*$-algebras $C^*(\Gxx)$ of the isotropy groups satisfying certain conditions.

The present paper is motivated by the natural question of how to describe the KMS states on the reduced groupoid C$^*$-algebras $C^*_r(\G)$. In other words, the question is under which conditions on $(\mu,\{\varphi_x\}_x)$ the corresponding KMS state on $C^*(\G)$ factors through $C^*_r(\G)$. A sufficient condition, which follows immediately from the construction, is that $\mu$-almost all traces $\varphi_x$ factor through~$C^*_r(\Gxx)$. But this condition is not in general necessary, since by an example of Willett~\cite{W} based on the HLS groupoids~\cite{HLS} there exist groupoids $\G$ such that $C^*(\G)=C^*_r(\G)$, yet $\G$ has a nonamenable isotropy group; see the related Example~\ref{ex:28a} below. In order to obtain a necessary and sufficient condition we introduce a new, in general exotic, C$^{*}$-norm $\lVert\cdot\rVert_e$ on the group algebras of the isotropy groups. Denoting by $C^*_e(\Gxx)$ the corresponding completions, we prove then that the state defined by $(\mu,\{\varphi_x\}_x)$ factors through $C^*_r(\G)$ if and only if $\mu$-almost all traces $\varphi_x$ factor through $C^*_e(\Gxx)$, see  Proposition~\ref{prop32}.

Although our initial motivation was to characterize KMS states on $C^*_r(\G)$, in the end the main focus of the paper is the norm $\lVert\cdot\rVert_e$ itself. In particular, we provide several sufficient conditions on the groupoid that ensure that this norm is equal to the reduced norm. It is easy to show that this is always the case for the transformation groupoids, see Proposition \ref{prop:transformation-groupoids}. We show that this is also often the case for graded groupoids, see Theorems \ref{thm:e=r} and~\ref{tmred} for the precise statements.

\smallskip

The paper consists of four sections. In Section \ref{section2} we recall basic properties of locally compact \'etale groupoids and their associated C$^*$-algebras and we prove a few auxiliary results. In Section \ref{section3} we define the C$^{*}$-norm $\lVert \cdot \rVert_{e}$ on the group algebras $\mathbb{C}\Gxx$ of the isotropy groups of a locally compact \'etale groupoid~$\mathcal{G}$. We prove that the norm $\lVert \cdot \rVert_{e}$ agrees with the reduced norm for transformation groupoids and, using~\citelist{\cite{HLS}\cite{W}}, we provide an example of a groupoid of germs where it is a genuine exotic norm. In Section \ref{section4} we prove that this norm completely governs which states on~$C^{*}(\G)$ with $C_{0}(\Gu)$ in their centralizers factor through $C_{r}^{*}(\G)$, and we give a partial extension of this result to weights. This covers the KMS states on the reduced groupoid C$^{*}$-algebras for the time evolutions defined by continuous real-valued $1$-cocycles. In Section~\ref{section5} we give two sufficient conditions for the exotic norm to agree with the reduced norm. A common assumption in both cases is the existence of a grading of the groupoid, so to illustrate the relevance of our results, we begin Section \ref{section5} by discussing several examples of graded groupoids. We end the paper by showing that the property $\lVert\cdot\rVert_e=\lVert\cdot\rVert_r$ is invariant under Morita equivalence.

As a last remark, let us comment on our assumptions on the groupoids. Since the groupoids of germs and the groupoids associated with semigroups are increasingly popular and can easily be non-Hausdorff, we work with not necessarily Hausdorff locally compact \'etale groupoids throughout the paper, see Section~\ref{section2} for the precise setup. This requires some extra care at a few places, but no fundamental changes compared to the Hausdorff case. Apart from Section~\ref{section4}, we do not assume that our groupoids are second countable either.

\smallskip\noindent
\emph{Acknowledgement}: We thank the referee for carefully reading the manuscript and suggesting to look at Morita equivalent groupoids.

\section{Preliminaries} \label{section2}

For a groupoid $\G$, we denote by $\Gu \subset \G$ its unit space and by $r,s\colon \G \to \Gu $ the range map and the source map, respectively, so that $r(g)=gg^{-1}$ and $s(g)=g^{-1}g$. For $x,y\in \Gu $, we set $\G_{x}:=s^{-1}(x)$, $\G^{x}:=r^{-1}(x)$ and $\G^y_x:=\G^y\cap\G_x$. In particular, $\G_{x}^{x}$ denotes the isotropy group at $x$.

\smallskip

We will be working with locally compact, but not necessarily Hausdorff, \'{e}tale groupoids, by which we mean groupoids $\G$ endowed with a topology such that:
\begin{enumerate}
\item[-] the groupoid operations are continuous;
\item[-] the unit space $\Gu $ is a locally compact Hausdorff space in the relative topology;
\item[-] the map $r$ is a local homeomorphism.
\end{enumerate}
These assumptions imply that the map $s$ is a local homeomorphism as well, the sets $\G^x$ and $\G_x$ are discrete (in particular, Hausdorff), and every point of $\G$ has a compact Hausdorff neighbourhood. It is known and not difficult to see that they also imply that $\Gu $ is open in $\G$.

If $W\subset \G$ is open and $r|_{W}\colon W\to r(W)$ and $s|_{W}\colon W\to s(W)$ are homeomorphisms onto the open sets $r(W)$ and $s(W)$, then $W$ is called an (open) bisection. Given two subsets $U$ and $V$ of $\G$, we denote by $UV$ the set of pairwise products. If $U$ and $V$ are bisections, then so is $UV$.

Recall how to construct the full and reduced groupoid C$^{*}$-algebras of $\G$. For an open Hausdorff subset $U\subset\G$, consider the space $C_c(U)$ of continuous compactly supported functions on $U$. We can consider every $f\in C_c(U)$ as a function on $\G$ by continuing $f$ by zero. Define the function space $C_c(\G)$ as the sum of the spaces $C_c(U)$ for all~$U$. Note that the functions in $C_c(\G)$ are not necessarily continuous. Later we will need the following simple lemma, proved using a partition of unity argument.

\begin{lemma}[\cite{KhSk}*{Lemma~1.3}]\label{lem:span}
If $(U_i)_{i\in I}$ is a covering of $\G$ by open Hausdorff sets, then $C_c(\G)$ is the sum of the spaces $C_c(U_i)$, $i\in I$.
\end{lemma}

We make $\CcG$ into a $*$-algebra by defining the convolution product $f_{1} * f_{2}$ of two functions $f_{1}, f_{2} \in C_{c}(\G)$ via the formula
\begin{equation*} \label{eprod}
(f_{1}*f_{2})(g) := \sum_{h \in \G^{r(g)}} f_{1}(h) f_{2}(h^{-1}g)\quad \text{for}\quad g\in \G,
\end{equation*}
and the involution by $f^{*}(g):=\overline{f(g^{-1})}$. We define a norm $\lVert \cdot \rVert $ on $\CcG$ by setting
\begin{equation*}
\lVert f \rVert := \sup_\rho \lVert \rho(f) \rVert,
\end{equation*}
where the supremum is taken over all representations of $C_c(\G)$ by bounded operators on Hilbert spaces. Completing $\CcG$ with this norm we obtain the \emph{full groupoid $C^{*}$-algebra} $C^{*}(\G)$ of $\G$.
%We call the norm $\lVert \cdot \rVert $ the full norm.

To introduce the reduced norm on $C_{c}(\G)$, for every point $x\in \Gu $ define a representation $\rho_{x}\colon C_{c}(\G) \to B(\ell^{2}(\G_{x}))$ by
\begin{equation}\label{eq:rhox}
\rho_{x}(f) \delta_{g} :=\sum_{h \in \G_{r(g)}} f(h) \delta_{hg},
\end{equation}
where $\delta_g$ is the Dirac delta-function. The reduced norm $\lVert \cdot \rVert_{r}$ on $C_{c}(\G)$ is defined by
\begin{equation*}
\lVert f \rVert_{r}  := \sup_{x\in \Gu } \lVert \rho_{x}(f) \rVert,
\end{equation*}
and the \emph{reduced groupoid $C^{*}$-algebra} $C^{*}_{r}(\G)$ is then the completion of $C_{c}(\G)$ with respect to this norm. The identity map $C_c(\G)\to C_c(\G)$ extends to a surjective $*$-homomorphism $C^{*}(\G)\to C^{*}_{r}(\G)$.

For all $f\in \CcG$, we have the inequalities $\lVert f\rVert_{\infty} \leq \lVert f\rVert_{r} \leq  \lVert f\rVert$, where $\lVert \cdot \rVert_{\infty}$ denotes the supremum-norm. If $f\in C_c(U)$ for a bisection $U$, then
\begin{equation} \label{eq:norm-bisection}
\|f\|=\lVert f\rVert_r=\lVert f\rVert_{\infty}.
\end{equation}
It follows from this and the definition of the product \eqref{eprod} that the space $C_{0}(\Gu)$ with its usual C$^{*}$-algebra structure is embedded into $C_{r}^{*}(\G)$ and $C^{*}(\G)$.

\smallskip

Next, recall the definition of induced representations. Take $x\in \Gu $ and consider the full group C$^*$-algebra $C^*(\Gxx)$ of the isotropy group $\Gxx$. Denote by $u_{g}$, $g\in \G_{x}^{x}$, the canonical unitary generators of $C^{*}(\G_{x}^{x})$. Assume that $\rho\colon  C^*(\G^{x}_{x}) \to B(H)$ is a representation. Let $L$ be the space of functions $\xi \colon  \G_{x} \to H$ satisfying
\begin{equation*}
\xi(gh)=\rho(u_{h}^{*})\xi(g)
\end{equation*}
for all $g \in \G_{x}$ and $h\in \G_{x}^{x}$ and such that
\begin{equation*}
\sum_{g\in \G_{x}/\G_{x}^{x}}\lVert \xi(g)\rVert^{2}<\infty .
\end{equation*}
We define a representation $\Ind\rho\colon C^*(\G)\to B(L)$ by
\begin{equation*}
((\Ind \rho)(f)\xi)(g) :=
\sum_{h\in \G^{r(g)}}f(h) \xi(h^{-1}g) \quad\text{for}\quad f\in C_c(\G).
\end{equation*}

If $\rho=\lambda_{\Gxx}$, the left regular representation of $\Gxx$, then $\Ind\lambda_{\Gxx}$ is unitarily equivalent to the representation $\rho_x\colon C^*(\G)\to B(\ell^2(\G_x))$ defined by~\eqref{eq:rhox}. Explicitly, in this case we have $L=\ell^2(\G_x\times_{\Gxx}\Gxx)$, where $\G_x\times_{\Gxx}\Gxx$ is the quotient of $\G_x\times\Gxx$ by the equivalence relation $(gh',h)\sim (g,h'h)$  ($g\in\G_x$, $h,h'\in\Gxx$), and the canonical bijection $\G_x\times_{\Gxx}\Gxx\to\G_x$, $(g,h)\mapsto gh$, gives rise to a unitary intertwiner between $\Ind\lambda_{\Gxx}$ and~$\rho_x$.

For a general representation $\rho$ of $C^*(\Gxx)$, we have a coisometric map
$$
v\colon L\to H,\qquad v\xi=\xi(x),
$$
with the adjoint given by
\begin{equation}\label{eq:projection}
(v^*\zeta)(g)=\begin{cases}\rho(u_g)^*\zeta,&\text{if}\ g\in\Gxx,\\ 0,&\text{otherwise}.\end{cases}
\end{equation}

Denote by $\eta_x$ the restriction map $C_{c}(\G) \ni f \mapsto f|_{\G_{x}^{x}}\in C_{c}(\G_{x}^{x})$. We then have
\begin{equation}\label{eq:projection2}
v(\Ind\rho)(f)v^*=\rho(\eta_x(f))
\end{equation}
for all $f\in C_c(\G)$. Taking a faithful representation and then the regular representation of~$C^*(\Gxx)$ as~$\rho$, we get the following result.

\begin{lemma}\label{lem:contraction}
The restriction map $\eta_x\colon C_c(\G)\to C_c(\Gxx)$ extends to a completely positive contraction
$$
\vartheta_x\colon C^*(\G)\to C^*(\Gxx),
$$
as well as to a completely positive contraction $\vartheta_{x,r}\colon C^*_r(\G)\to C^*_r(\Gxx)$.
\end{lemma}

The states $\varphi\circ\vartheta_x$ are of particular interest to us. Their GNS-representations are described as follows.

\begin{lemma}\label{lem:GNS}
Assume $\varphi$ is a state on $C^*(\Gxx)$ and $(H_\varphi,\pi_\varphi,\xi_\varphi)$ is the associated GNS-triple. Consider the induced representation $\Ind\pi_\varphi\colon C^*(\G)\to B(L)$. Then $(L,\Ind\pi_\varphi,v^*\xi_\varphi)$ is a GNS-triple associated with $\varphi\circ\vartheta_x$, where $v^*$ is the isometry~\eqref{eq:projection}.
\end{lemma}

\bp Identity~\eqref{eq:projection2} for $\rho=\pi_\varphi$ implies that
$$
\varphi\circ\vartheta_x=((\Ind\pi_\varphi)(\cdot)v^*\xi_\varphi,v^*\xi_\varphi).
$$
Therefore we only need to check that the vector $v^*\xi_\varphi$ is cyclic.

It suffices to show that if $\xi\in L$ is nonzero, then there exists $f\in C_c(\G)$ such that
$$
(\xi,(\Ind\pi_\varphi)(f)v^*\xi_\varphi)\ne0.
$$
Let $g\in\G_x$ be such that $\xi(g)\ne0$. There exists $h\in \Gxx$ such that $(\xi(g),\pi_\varphi(u_{h})\xi_\varphi)\ne0$. Choose a bisection $U$ containing $gh$ and pick $f\in C_{c}(U)$ with $f(gh)=1$. Since $U \cap \G_{x}=\{gh\}$ we have $f(gh)=1$ and $f=0$ on $\G_{x}\setminus\{gh\}$. Then
$$
((\Ind \pi_\varphi)(f)v^*\xi_\varphi)(g)=\sum_{g' \in \G^{x}} f(gg')(v^*\xi_\varphi)(g'^{-1})
=f(gh)(v^*\xi_\varphi)(h^{-1}) = \pi_\varphi(u_{h})\xi_\varphi ,
$$
and similarly $(\Ind \pi_\varphi)(f)v^*\xi_\varphi=0$ on $\G_x\setminus g\Gxx$. It follows that
$$
(\xi,(\Ind\pi_\varphi)(f)v^*\xi_\varphi)=(\xi(g),\pi_\varphi(u_{h})\xi_\varphi)\ne0,
$$
as needed.
\ep

The maps $\vartheta_x$ and $\vartheta_{x,r}$ have large multiplicative domains. Specifically, we have the following elementary result, which we will use repeatedly throughout the paper.

\begin{lemma}\label{lem21}
Assume that $g_{1}, g_{2}, \dots, g_{n}$ are distinct points in $\G_{x}^{x}$ and $\{U_{i}\}_{i=1}^{n}$ is a family of bisections with $g_{i}\in U_{i}$ for each $i$.
Assume $f \in C_{c}(\G)$ is zero outside $\bigcup_{i=1}^{n} U_{i}$. Then $f$ lies in the multiplicative domains of $\vartheta_x$  and $\vartheta_{x,r}$.
\end{lemma}

\begin{proof}
Take $f'\in C_c(\G)$. For $g\in\Gxx$, we have that
\begin{equation*}
(f*f')(g) = \sum_{h \in \G^{x}} f(h) f'(h^{-1}g) .
\end{equation*}
If $h \in \G^{x}$ and $f(h)\neq 0$, we have $h\in U_{j}$ for some $j$. Since $g_{j}\in U_{j}$ and $r(g_{j})=x$ we conclude that $h=g_{j}\in \G_{x}^{x}$. This implies that
\begin{equation*}
(f*f')(g) = \sum_{h \in \G_{x}^{x}} f(h) f'(h^{-1}g)=(\eta_{x}(f)*\eta_{x}(f'))(g),
\end{equation*}
hence $\eta_x(f*f')=\eta_x(f)*\eta_x(f')$. In a similar way we check that $\eta_x(f'*f)=\eta_x(f')*\eta_x(f)$.
\end{proof}

\section{A possibly exotic norm on the group algebras of isotropy groups}\label{section3}

Assume that we are given a locally compact \'etale groupoid $\G$ and a point $x\in \Gu $. We define a C$^*$-norm on the group algebra $\C\Gxx$ as follows.

\begin{defn}\label{def}
For $h\in C_c(\Gxx)$, let
\begin{equation}\label{eq:e-norm}
\|h\|_e:=\sup_{\rho\in \mathcal S_x}\|\rho(h)\|,
\end{equation}
where $\mathcal S_x$ is the collection of representations $\rho$ of $C^*(\Gxx)$ such that $\Ind\rho$ factors through $C^*_r(\G)$.
\end{defn}

Since $\mathcal S_x$ contains the regular representation $\lambda_{\Gxx}$, this is indeed a C$^*$-norm and $\lVert\cdot\rVert_e\ge\lVert\cdot\rVert_r$. Denote by $C^*_e(\Gxx)$ the completion of $\C\Gxx$ with respect to this norm.

\begin{prop}\label{prop:e-charact}
For any representation $\rho$ of $C^*(\Gxx)$, the representation $\Ind\rho$ of~$C^*(\G)$ factors through~$C^*_r(\G)$ if and only if $\rho$ factors through $C^*_e(\Gxx)$.
\end{prop}

\bp
If $\Ind\rho$ factors through~$C^*_r(\G)$, then by definition $\rho\in \mathcal S_x$ and hence $\rho$ factors through $C^*_e(\Gxx)$. Conversely, assume $\rho$ factors through $C^*_e(\Gxx)$. Since the collection $\mathcal S_x$ is closed under direct sums, we can find $\pi\in \mathcal S_x$ such that $\|h\|_e=\|\pi(h)\|$ for all $h\in\C\Gxx$. Then $\rho$ is weakly contained in $\pi$. Since weak containment is preserved under induction, we conclude that $\rho\in\mathcal S_x$.
\ep

The norm $\lVert\cdot\rVert_e$ can also be defined in terms of states as follows.

\begin{prop}\label{prop:e-norm-states}
A state $\varphi$ on $C^*(\Gxx)$ factors through $C^*_e(\Gxx)$ if and only if the state $\varphi\circ\vartheta_x$ on~$C^*(\G)$ factors through $C^*_r(\G)$. Hence, for every $h\in C_c(\Gxx)$, we have
$$
\|h\|_e=\sup\varphi(h^**h)^{1/2},
$$
where the supremum is taken over all states $\varphi$ on $\C\Gxx$ such that $\varphi\circ\eta_x$ is bounded with respect to the reduced norm on $C_c(\G)$.
\end{prop}

\bp
Since a state on a C$^*$-algebra vanishes on a closed ideal if and only if the associated GNS-representation vanishes on the same ideal, the result follows from Lemma~\ref{lem:GNS} and Proposition~\ref{prop:e-charact}.
\ep

Similarly to the existence of contractions $\vartheta_x\colon C^*(\G)\to C^*(\Gxx)$ and $\vartheta_{x,r}\colon C^*_r(\G)\to C^*_r(\Gxx)$, identity~\eqref{eq:projection2} for any faithful representation $\rho$ of $C^*_e(\Gxx)$ shows that $\eta_x$ extends to a completely positive contraction
$$
\vartheta_{x,e}\colon C^*_r(\G)\to C^*_e(\Gxx).
$$
This contraction is surjective and then as a Banach space $C^*_e(\Gxx)$ is isometrically isomorphic to the quotient space $C^*_r(\G)/\ker \vartheta_{x,e}$. Indeed, if $A\subset C^*_r(\G)$ is the multiplicative domain of $\vartheta_{x,e}$, then by Lemma~\ref{lem21} the image of the C$^*$-algebra $A$ in $C^*_e(\Gxx)$ under $\vartheta_{x,e}$ is dense, hence it coincides with~$C^*_e(\Gxx)$ and the norm on $C^*_e(\Gxx)\cong A/\ker(\vartheta_{x,e}|_A)$ is the quotient norm. Since $\vartheta_{x,e}$ is a contraction, we then conclude that the norm $\lVert\cdot\rVert_e$ can also be described as the quotient norm on the Banach space $C^*_r(\G)/\ker \vartheta_{x,e}$, as claimed.

From the practical point of view this is still not very useful, as it is not clear what the kernel of~$\vartheta_{x,e}$ is. The following theorem sharpens the above observation and will allow us to describe the kernel.

\begin{thm}\label{thmnorm}
For any locally compact \'etale groupoid $\G$, $x\in \Gu$ and $h\in C_c(\Gxx)$, we have
\begin{equation}\label{eq:norm-inf}
\lVert h \rVert_{e} = \inf\left\{ \lVert f \rVert_{r} : f\in C_{c}(\G),\ \eta_{x}(f)=h    \right\}.
\end{equation}
Furthermore, let $\V$ be a neighbourhood base at $x$ partially ordered by containment. For each $V\in\V$, choose a function $q_V\in C_c(\Gu)$ such that $q_V(x)=1$, $0\le q_V\le 1$ and $\supp q_V\subset V$. For $h\in C_c(\Gxx)$, choose any function $f\in C_c(\G)$ such that $\eta_x(f)=h$. Then
\begin{equation}\label{eq:e-norm-limit}
\|h\|_e=\lim_{V\in\V}\|q_V*f*q_V\|_r.
\end{equation}
\end{thm}

We divide the proof of the theorem into several lemmas. Denote the right hand side of~\eqref{eq:norm-inf} by~$\|h\|_e'$.

\begin{lemma}\label{lem:e-expression}
If $\eta_x(f)=h$, then the limit in \eqref{eq:e-norm-limit} exists and equals $\|h\|_e'$.
\end{lemma}

\bp Let us show first that if $f\in C_c(\G)$ satisfies $\eta_x(f)=0$, then
\begin{equation}\label{eq:f0}
\lim_{V\in\V}\|q_V*f*q_V\|_r=0.
\end{equation}
%This is essentially obvious in the Hausdorff case, but requires some work for non-Hausdorff groupoids.

By Lemma~\ref{lem:span} we can write $f=\sum^m_{i=1}f_i$, with $f_i\in C_c(U_i)$ for a bisection $U_i$. Assume
$$
\Gxx\cap\Big(\bigcup^m_{i=1} U_i\Big)=\{g_1,\dots,g_n\}.
$$

Take an index $i$. Assume first that $\Gxx\cap U_i=\emptyset$. In this case we eventually have $q_V*f_i*q_V=0$. Indeed, let $K_i\subset U_i$ be the support of $f_i|_{U_i}$. For every $g \in K_i$ we must have that either $r(g)\neq x$ or $s(g)\neq x$, so we can find a neighbourhood $W_g$ of $g$ in $U_i$ with either $x\notin\overline{r(W_g)}$ or $x\notin\overline{s(W_g)}$. By compactness of $K_i$ there is a neighbourhood $W$ of $x$ in $\Gu $ such that for each $g\in K_i$ we have either $r(g)\notin W$ or $s(g)\notin W$. This implies that if $x\in V\subset W$, then $q_V*f_{i}*q_V=0$, proving our claim.

Consider now indices $i$ such that $\Gxx\cap U_i\ne\emptyset$. For every such $i$ there is a unique index $k_i$ such that $\Gxx\cap U_i=\{g_{k_i}\}$. Put
$$
W_k=\bigcap_{i:k_i=k}U_i.
$$
For every $k$, the set $W_k$ is a bisection containing $g_k$ and we have
\begin{equation}\label{eq:f(gk)}
\sum_{i:k_i=k}f_i(g_k)=f(g_k)=0.
\end{equation}

Since $(r^{-1}(V)\cap U_i)_{V\in\V}$ is a neighbourhood base for $g_{k_i}$ in $U_i$ and $f_i|_{U_i}$ is continuous at $g_{k_i}$, we have
$$
\lim_V\|q_V*f_i*q_V-f_i(g_{k_i})q_V*\un_{W_{k_i}}*q_V\|_\infty=0.
$$
For $x\in V\subset \overline{V}\subset r(W_{k_i})$, we have $q_V*f_i*q_V-f_i(g_{k_i})q_V*\un_{W_{k_i}}*q_V\in C_c(U_i)$. Hence, by~\eqref{eq:norm-bisection}, we may conclude that
$$
\lim_V\|q_V*f_i*q_V-f_i(g_{k_i})q_V*\un_{W_{k_i}}*q_V\|_r=0,
$$
which together with~\eqref{eq:f(gk)} gives
$$
\lim_V\Big\|\sum_{i:k_i=k}q_V*f_i*q_V\Big\|_r=0
$$
for all $k=1,\dots,n$. This, combined with the fact that we eventually have $q_V*f_i*q_V=0$ for $i$ such that $\Gxx\cap U_i=\emptyset$, proves~\eqref{eq:f0}.

\smallskip

Assume now that $\eta_x(f)=h$. For every $W\in \V$, we have
$$
\lim_V\|q_V*f*q_V-(q_Vq_W)*f*(q_Wq_V)\|_r=0
$$
by~\eqref{eq:f0} applied to $f-q_W*f*q_W$ or simply because $\|q_V-q_Wq_V\|_\infty\to0$.
As $\|(q_Vq_W)*f*(q_Wq_V)\|_r\le\|q_W*f*q_W\|_r$, this implies that
$$
\limsup_{V}\|q_V*f*q_V\|_r\le\|q_W*f*q_W\|_r.
$$
Since this is true for all $W\in\V$, we conclude that the limit in~\eqref{eq:e-norm-limit} indeed exists. Furthermore, since $\eta_x(q_V*f*q_V)=\eta_x(f)=h$ and $\|q_V*f*q_V\|_r\le\|f\|_r$, we have
\begin{equation}\label{eq:e-norm-limit1}
\|h\|_e'\le \lim_{V}\|q_V*f*q_V\|_r\le\|f\|_r.
\end{equation}

Now, take another function $f'\in C_c(\G)$ such that $\eta_x(f')=h$. By~\eqref{eq:f0} applied to $f-f'$ we conclude that
the limit of $\|q_V*f*q_V\|_r$ is independent of $f$ such that $\eta_x(f)=h$. Together with the inequalities~\eqref{eq:e-norm-limit1} and the definition of $\|h\|_e'$, this proves that this limit is $\|h\|_e'$.
\ep

It is clear that $\lVert\cdot\rVert_e'$ is a seminorm on $C_c(\Gxx)$, but we can now say much more.

\begin{lemma}\label{lem:c-star-norm}
The seminorm  $\lVert \cdot \rVert_{e}'$ is a C$^*$-norm on the group algebra $\C\Gxx$, and $\lVert \cdot \rVert_{e}'\geq \lVert \cdot \rVert_{e}$.
\end{lemma}

\begin{proof}
Since $\vartheta_{x,e}\colon C^*_r(\G)\to C^*_e(\Gxx)$ is a contraction, if $h\in C_{c}(\G_{x}^{x})$ and $f\in C_{c}(\G)$ satisfy $\eta_{x}(f)=h$, we get
$$
\|h\|_e=\|\eta_x(f)\|_e\le\|f\|_r.
$$
By the definition of $\lVert \cdot \rVert_e'$ it follows that $\|h\|_e\le\|h\|_e'$. In particular, $\lVert\cdot\rVert_e'$ is a norm.

Next, let us check that $\|h_1*h_2\|_e'\le\|h_1\|_e'\|h_2\|_e'$. Fix $\eps>0$ and choose $f_1\in C_c(\G)$ such that $\eta_x(f_1)=h_1$ and $\|f_1\|_r<\|h_1\|_e'+\eps$. Suppose $h_2$ is supported on $\{g_1,\dots,g_n\}\subset\Gxx$. Let $U_1,\dots,U_n$ be bisections such that $\Gxx\cap U_i=\{g_i\}$. Choose $\varphi_i\in C_c(U_i)$ such that $\varphi_i(g_i)=1$ and let $f_2=\sum_i h_2(g_i)\varphi_i$. Then $\eta_x(f_2)=h_2$ and by Lemma~\ref{lem21} we have $\eta_x(f_1*f_2)=\eta_x(f_1)*\eta_x(f_2)=h_1*h_2$. By Lemma~\ref{lem:e-expression}, by replacing $\varphi_i$ with $q_V*\varphi_i*q_V$ (where $q_V$ is as in Theorem~\ref{thmnorm}), we may assume that $\|f_2\|_r<\|h_2\|_e'+\eps$.  Hence
$$
\|h_1*h_2\|_e'\le\|f_1*f_2\|_r\le(\|h_1\|_e'+\eps)(\|h_2\|_e'+\eps),
$$
which implies that $\|h_1*h_2\|_e'\le\|h_1\|_e'\|h_2\|_e'$, as $\eps>0$ was arbitrary.

In order to show that $\lVert\cdot\rVert_e'$ is a C$^*$-norm, it remains to check that $\|h\|_e'^2\le\|h^**h\|_e'$. Similarly to the previous paragraph, by Lemma~\ref{lem21} we can find $f\in C_c(\G)$ such that $\eta_x(f)=h$ and $\eta_x(f^**f)=h^**h$. Then we have
\begin{align*}
\|h\|_e'^2&=\lim_{V}\|q_V*f*q_V\|_r^2= \lim_{V}\|q_V*f^**q_V^2*f*q_V\|_r\\
&\le \lim_{V}\|q_V*f^**f*q_V\|_r=\|h^**h\|_e',
\end{align*}
which completes the proof of the lemma.
\ep

\bp[Proof of Theorem~\ref{thmnorm}]
We already know that $\lVert\cdot\rVert_e\le\lVert\cdot\rVert_e'$. In order to prove the opposite inequality it suffices to show that if $\varphi$ is a state on $\C\Gxx$ bounded with respect to $\lVert\cdot\rVert_e'$, then it is also bounded with respect to $\lVert\cdot\rVert_e$. By Proposition~\ref{prop:e-norm-states} this means that we have to show that $\varphi\circ\eta_x$ is bounded with respect to the reduced norm on $C_c(\G)$. But this is obvious, since by definition $\lVert\cdot\rVert_e'$ is the quotient reduced norm on $C_c(\G)/\ker\eta_x$.
\ep

\begin{cor}
The space $\ker(\eta_x\colon C_c(\G)\to C_c(\Gxx))$ is dense in $\ker(\vartheta_{x,e}\colon C^*_r(\G)\to C^*_e(\Gxx))$.
\end{cor}

\bp
Take $a\in\ker\vartheta_{x,e}$ and $\eps>0$. We can find $f_1\in C_c(\G)$ such that $\|a-f_1\|_r<\eps$. Then
$$
\|\eta_x(f_1)\|_e=\|\vartheta_{x,e}(f_1-a)\|_e<\eps.
$$
By Theorem~\ref{thmnorm} it follows that we can find $f_2\in C_c(\G)$ such that $\eta_x(f_2)=\eta_x(f_1)$ and $\|f_2\|_r<\eps$. Then for $f:=f_1-f_2\in C_c(\G)$ we have $\eta_x(f)=0$ and $\|a-f\|_r\le\|a-f_1\|_r+\|f_2\|_r<2\eps$.
\ep

\begin{cor}\label{cor:exact-sequence}
If $\{x\}\subset\Gu$ is an invariant subset, so that $\G^y_x=\emptyset$ for all $y\ne x$, then $\G\setminus\Gxx$ is a locally compact \'etale groupoid and we have a short exact sequence
\begin{equation}\label{eq:exact-sequence}
0\to C^*_r(\G\setminus\Gxx)\to C^*_r(\G)\to C^*_e(\Gxx)\to0.
\end{equation}
\end{cor}

\bp
It is easy to see that $\G\setminus\Gxx$ is an open subgroupoid of $\G$.
As $\vartheta_{x,e}\colon C^*_r(\G)\to C^*_e(\Gxx)$ is a $*$-homomorphism in the present case and $\ker\eta_x$ is dense in $\ker\vartheta_{x,e}$, we then only need to show that $C_c(\G\setminus\Gxx)$ is dense in $\ker\eta_x$ with respect to the reduced norm on $C_c(\G)$. But this follows from~\eqref{eq:f0}, since $f-q_V*f*q_V\in C_c(\G\setminus\Gxx)$ for any $f\in C_c(\G)$ as long as $q_V=1$ in a neighbourhood of $x$.
\ep

\begin{remark}
If one is interested only in the setting of Corollary~\ref{cor:exact-sequence}, then it is actually easier to first realize that there is a C$^*$-norm $\lVert\cdot\rVert_e$ on $\C\Gxx$ making~\eqref{eq:exact-sequence} exact, and then that this norm must satisfy~\eqref{eq:e-norm} and~\eqref{eq:norm-inf}, cf.~\cite{Rbook}*{Proposition~II.4.5} (but note that this proposition erroneously claims that we get the reduced norm) and \cite{W}*{Lemma~2.7}. Similar statements can be proved for any closed invariant subset of $\Gu$ in place of~$\{x\}$.
\hfill$\diamondsuit$
\end{remark}

As a first example consider the transformation groupoids. Assume $X$ is a locally compact Hausdorff space and a discrete group $\Gamma$ acts on $X$ by homeomorphisms. The transformation groupoid $\G:=\Gamma\ltimes X$ corresponding to such an action is the set $\Gamma\times X$ with product
\begin{equation*}
(g,hx)(h,x)=(gh,x),\quad\text{so that}\quad r(g,x)=gx,\quad s(g,x)=x\quad \text{and}\quad (g,x)^{-1}=(g^{-1}, gx).
\end{equation*}
Endowed with the product topology, it becomes a locally compact Hausdorff \'etale groupoid, and then $C^*_r(\Gamma)\cong C_0(X)\rtimes_r\Gamma$. For the isotropy groups we have $\Gxx=\Gamma_x\times\{x\}\cong\Gamma_x$, where $\Gamma_x$ is the stabilizer of $x$.

\begin{prop}\label{prop:transformation-groupoids}
For any transformation groupoid $\G=\Gamma\ltimes X$, the norm $\lVert\cdot\rVert_e$ on $\C\G^x_x\cong\C \Gamma_x$ coincides with the reduced norm.
\end{prop}

\bp
We have an isometric embedding of $C^*_r(\Gamma_x)$ into $M(C_0(X)\rtimes_r \Gamma)$. Using this embedding, for any $a\in\C \Gamma_x$ and $f\in C_c(X)$ such that $\|f\|_\infty=1$ and $f(x)=1$, we have $af\in C_c(\G)$ and $\eta_x(af)=a$. It follows that $\|a\|_e\le\|af\|_r\le\|a\|_r$. Hence $\|a\|_e=\|a\|_r$.
\ep

Examples where the norm $\lVert\cdot\rVert_e$ is exotic (that is, it is neither reduced nor maximal~\cite{KS}) are, however, just one step away from the transformation groupoids. Namely, they can be found among the associated groupoids of germs, which are defined as follows.

Given an action of a discrete group $\Gamma$ on a locally compact Hausdorff space $X$, define an equivalence relation on $\Gamma\times X$ by saying $(g,x)\sim(h,x)$ if $gy=hy$ for all $y$ in a neighbourhood of $x$. Then $\G=(\Gamma\times X)/{\sim}$ is a groupoid with the unit space~$X$ and the composition $[g,hx][h,x]=[gh,x]$, where $[g,x]$ denotes the equivalence class of $(g,x)$. Equipped with the quotient topology, $\G$ becomes a locally compact \'etale groupoid. In general, such groupoids are non-Hausdorff.

It follows from examples of Higson--Lafforgue--Skandalis~\cite{HLS}*{Section~2}, see also~\cite{W}, that the norm $\lVert\cdot\rVert_e$ can be exotic for groupoids that are group bundles. In the next example we make a minimal modification of their construction to demonstrate the same phenomenon for groupoids of germs.

\begin{example} \label{ex28}
Let $\Gamma$ be a discrete group and $(\Gamma_n)^\infty_{n=1}$ be a decreasing sequence of finite index normal subgroups of $\Gamma$ such that $\bigcap^\infty_{n=1}\Gamma_n=\{e\}$. Let $X_n:=\Gamma/\Gamma_n$ and view $X_n$ as a finite discrete set. We have an action of $\Gamma$ on $X_n$ by translations. Let $X$ be the one-point compactification of $\bigsqcup^\infty_{n=1} X_n$. The actions of $\Gamma$ on $X_n$ together with the trivial action on the point $\infty\in X$ define an action of $\Gamma$ on $X$.

Let $\G$ be the corresponding groupoid of germs. Explicitly, we have
$$
\G=\bigsqcup^{\infty}_{n=1}((\Gamma/\Gamma_n)\times X_n)\sqcup (\Gamma\times\{\infty\}),
$$
each finite set $(\Gamma/\Gamma_n)\times X_n$ is open and has the discrete topology, while the sets $\cup_{n\ge N}(\{\pi_n(g)\}\times X_n)\cup\{(g,\infty)\}$ for $N\in\mathbb N$ form a neighbourhood base at $(g,\infty)$, where $\pi_n\colon\Gamma\to\Gamma/\Gamma_n$ is the quotient map. Note that $\G$ is Hausdorff.

Consider the point $x=\infty\in X=\Gu$, which is the only point of $\Gu$ with nontrivial isotropy. Then $\Gxx=\Gamma\times\{\infty\}\cong\Gamma$. Using~\eqref{eq:e-norm-limit} it is easy to see that, for all $a\in\C\Gamma$,
\begin{equation}\label{eqelimes}
\|a\|_e=\lim_{N\to\infty}\max\big\{\sup_{n\ge N}\|\lambda_n(a)\|,\|a\|_r\big\}=\lim_{n\to\infty}\|\lambda_n(a)\|,
\end{equation}
where $\lambda_n$ denotes the quasi-regular representation of $\Gamma$ on $\ell^2(\Gamma/\Gamma_n)$ and the second equality above follows because the sequence $\{\|\lambda_n(a)\|\}_n$ is nondecreasing and the regular representation of $\Gamma$ is weakly contained in $\bigoplus_{n}\lambda_n$, so that $\|a\|_r\le \sup_{n}\|\lambda_n(a)\|$, cf.~\cite{W}*{Lemma~2.7}.

Following~\citelist{\cite{HLS}\cite{W}}, an explicit example with $\lVert\cdot\rVert_e\ne\lVert\cdot\rVert_r,\lVert\cdot\rVert_{\mathrm{max}}$ can be obtained as follows. Take $\Gamma=SL_2(\Z)$ and $\Gamma_n=\ker(SL_2(\Z)\to SL_2(\Z/2^n\Z))$. Since $\Gamma$ is nonamenable and the trivial representation of $\Gamma$ is contained in $\lambda_n$ for every $n$, we cannot have $\lVert\cdot\rVert_e=\lVert\cdot\rVert_r$. On the other hand, the trivial representation is isolated in the irreducible unitary representations weakly contained in~$\bigoplus^\infty_{n=1}\lambda_n$ by Selberg's theorem, yet $\Gamma$ does not have property~(T). Hence we cannot have $\lVert\cdot\rVert_e=\lVert\cdot\rVert_{\mathrm{max}}$ either.
%\hfill$\diamondsuit$
\end{example}

\section{Application to states and weights with diagonal centralizers} \label{section4}

Throughout this section we assume that $\G$ is a second countable locally compact \'etale groupoid. (But we still do not require $\G$ to be Hausdorff.)

\smallskip

Consider a Radon measure $\mu$ on $\Gu$. A \emph{$\mu$-measurable field of states} on the C$^*$-algebras of the isotropy groups is a collection $\{\varphi_{x}\}_{x\in \Gu }$, where $\varphi_{x}$ is a state on $C^{*}(\G_{x}^{x})$ for all $x\in \Gu$, such that the function
\begin{equation*}
\Gu \ni x \mapsto \sum_{g\in \G_{x}^{x}} f(g) \varphi_{x}(u_{g})
\end{equation*}
is $\mu$-measurable for all $f\in C_{c}(\G)$, where $\{ u_{g}\}_{g\in \Gxx}$ denotes the canonical unitary generators of $C^{*}(\G_{x}^{x})$ for $x\in \Gu$.

By \cite{N}*{Theorem~1.1} \label{tNesh}, there is a bijective correspondence between the states on $C^{*}(\G)$ containing $C_{0}(\Gu )$ in their centralizers and the pairs $(\mu, \{\varphi_{x}\}_{x\in \Gu })$ consisting of a (regular, Borel) probability measure $\mu$ on $\Gu $ and a $\mu$-measurable field of states $\{\varphi_{x}\}_{x\in \Gu}$. The state $\varphi$ corresponding to $(\mu, \{\varphi_{x}\}_{x\in \Gu })$ is given by
$$
\varphi(f)=\int_{\Gu } \sum_{g\in \G_{x}^{x}} f(g) \varphi_{x}(u_{g}) \,d\mu(x)
$$
for all $f\in C_{c}(\G)$. In other words,
\begin{equation}\label{eq:decomposition}
\varphi=\int_\Gu\varphi_x\circ\vartheta_x\,d\mu(x),
\end{equation}
with the integral understood in the weak$^*$ sense. To be precise, the formulation in \cite{N} requires $\G$ to be Hausdorff, but as has already been observed in \cite{NS}, the result is true in the non-Hausdorff case as well, with essentially the same proof; see also Remark~\ref{rem:decomposition} below.

\begin{prop}\label{prop32}
For any second countable locally compact \'etale groupoid $\G$, a state $\varphi$ on $C^{*}(\G)$ with centralizer containing $C_0(\Gu)$ factors through $C^*_r(\G)$ if and only if the state $\varphi_x$ factors through $C^*_e(\Gxx)$ for $\mu$-almost all $x$, where $(\mu, \{\varphi_{x}\}_{x\in \Gu})$ is the pair associated with $\varphi$.
\end{prop}

\bp
Let $I$ be the kernel of the canonical map $C^*(\G)\to C^*_r(\G)$. Since $\G$ is second countable, the C$^*$-algebra $C^*(\G)$ is separable. (Indeed, we can cover $\G$ by countably many bisections $U_n$, choose countable subsets of $C_c(U_n)$ that are dense in the supremum-norm, then the sums of finitely many elements in the union of all these subsets will be dense in $C^*(\G)$ by \eqref{eq:norm-bisection} and Lemma~\ref{lem:span}.) Let $F$ be a countable dense subset of~$I_+$. The state $\varphi$ factors through $C^*_r(\G)$ if and only if it vanishes on $F$. One the other hand, by Proposition~\ref{prop:e-norm-states}, the state $\varphi_x$ factors through $C^*_e(\Gxx)$ if and only if $\varphi_x\circ\vartheta_x$ vanishes on $F$. Since
$$
\varphi(a)=\int_\Gu(\varphi_x\circ\vartheta_x)(a)\,d\mu(x)
$$
for all $a\in F$, we get the result.
\ep

\begin{example}\label{ex:28a}
As discussed in the introduction, our main motivation for writing this paper was to characterize the KMS states on $C_{r}^{*}(\G)$ for the time evolutions defined by real-valued $1$-cocycles on~$\G$. A KMS state for such a time evolution on $C^{*}(\G)$ is always given by a pair $(\mu, \{\varphi_{x}\}_{x\in \Gu})$ satisfying some extra conditions \cite{N}*{Theorem 1.3}, and the question was when it factors through $C^*_r(\G)$.

Consider, for instance, the groupoid $\G$ from Example~\ref{ex28} defined by a group $\Gamma$ and a decreasing sequence $(\Gamma_n)^\infty_{n=1}$ of finite index normal subgroups. Take the zero cocycle, so we want to describe the tracial states on $C^*_r(\G)$. Then Proposition~\ref{prop32} and \cite{N}*{Theorem 1.3} imply that such traces are classified by pairs $(\mu,\tau)$, where $\mu$ is a probability measure on $X$ such that $\mu|_{X_n}$ is a scalar multiple of the counting measure for all $n$ and, if $\mu(\infty)>0$, $\tau$ is a tracial state on $C^*_e(\Gamma)=C^*_e(\G^\infty_\infty)$, otherwise $\tau$ is irrelevant. (This description also easily follows from Corollary~\ref{cor:exact-sequence}, since $C^*_r(\G\setminus\G^\infty_\infty)\cong c_0\text{-}\bigoplus^\infty_{n=1}B(\ell^2(\Gamma/\Gamma_n))$.)

Note that there can be many more tracial states on $C^*_e(\Gamma)$ than those that factor through $C^*_r(\Gamma)$, which are the only obvious ones that give rise to tracial states on $C^*_r(\G)$. Consider, for example, the free group $\Gamma=F_{2}$ on two generators, and define the subgroups $(\Gamma_{n})_{n=1}^{\infty}$ as in \cite{W}*{Lemma 2.8}. It then follows from \cite{W}*{Lemma 2.8} and \eqref{eqelimes} that $\lVert \cdot \rVert_{e}$ agrees with the full norm on $\C\Gamma=\C F_{2}$. Therefore there are plenty of tracial states on $C^*_e(\Gamma)$, but only one tracial state on $C^*_r(\Gamma)$, namely, the canonical trace.
\hfill$\diamondsuit$
\end{example}

For completeness, let us describe the GNS-representation defined by the state $\varphi$ given by a pair $(\mu, \{\varphi_{x}\}_{x\in \Gu})$. Let $(H_x,\pi_x,\xi_x)$ be the GNS-triple associated with $\varphi_x\circ\vartheta_x$. Recall that by Lemma~\ref{lem:GNS} the representation $\pi_x$ is induced from the GNS-representation defined by $\varphi_x$.

\begin{lemma}
There is a unique structure of a $\mu$-measurable field of Hilbert spaces on $(H_x)_{x\in\Gu}$ such that the sections $(\pi_x(f)\xi_x)_{x\in\Gu}$ are measurable for all $f\in C_c(\G)$. Namely, a section $(\zeta_x)_{x\in \Gu}$ is measurable if and only if the function $x\mapsto (\zeta_x,\pi_x(f)\xi_x)$ is $\mu$-measurable for all $f\in C_c(\G)$.
\end{lemma}

\bp
Choose a sequence $\{f_n\}_n$ of elements of $C_c(\G)$ that is dense in $C^*(\G)$. Then $\{\pi_x(f_n)\xi_x\}_n$ is dense in $H_x$ for every $x\in\Gu$. By \cite{Tak}*{Lemma~IV.8.10} it follows that there is a unique structure of a $\mu$-measurable field of Hilbert spaces on $(H_x)_{x\in\Gu}$ such that the sections $(\pi_x(f_n)\xi_x)_{x\in\Gu}$ are measurable for all $n$, namely, a section $(\zeta_x)_{x\in \Gu}$ is measurable if and only if the function $x\mapsto (\zeta_x,\pi_x(f_n)\xi_x)$ is $\mu$-measurable for all $n\in \mathbb{N}$. To finish the proof of the lemma it is then enough to show that the section $(\pi_x(f)\xi_x)_{x\in\Gu}$ is measurable for every $f\in C_c(\G)$. But this is clearly true, since we can find a subsequence $\{f_{n_k}\}_k$ converging to $f$ in $C^*(\G)$, and then $\pi_x(f_{n_k})\xi_x\to \pi_x(f)\xi_x$ for all $x\in\Gu$.
\ep

We can therefore consider the direct integral~$\pi$ of the GNS-representations $\pi_x$ on
$$
H:=\int^\oplus_{\Gu}H_x\,d\mu(x).
$$
Consider also the vector $\xi:=(\xi_x)_{x\in \Gu}\in H$.

\begin{prop}
With the above notation, $(H,\pi,\xi)$ is a GNS-triple associated with $\varphi$.
\end{prop}

\bp
We only need to show that the vector $\xi$ is cyclic. Consider the representation $m$ of $L^\infty(\Gu,\mu)$ on $H=\int^\oplus_{\Gu}H_x\,d\mu(x)$ by diagonal operators. As $\pi_x(f)\xi_x=f(x)\xi_x$ for all $f\in C_0(\Gu)$ and $x\in\Gu$, we then have
$$
\pi(f)\xi=m(f)\xi\quad\text{for all}\quad f\in C_0(\Gu).
$$
It follows that for all $a\in C^*(G)$ and $f\in C_0(\Gu)$ we have
$$
m(f)\pi(a)\xi=\pi(a)m(f)\xi=\pi(af)\xi.
$$
Therefore the subspace $\overline{\pi(C^*(\G))\xi}$ is invariant under the diagonal operators, hence the projection~$P$ onto this subspace is a decomposable operator. Thus, $P=(P_x)_{x\in \Gu}$ for a measurable family of projections $P_x$. For $\mu$-almost all $x$, we also have that the projection $P_x$ commutes with $\pi_x(C^*(G))$ and $P_x\xi_x=\xi_x$. Since the vector $\xi_x$ is cyclic, it follows that $P_x=1$ a.e., hence $P=1$.
\ep

\begin{remark}\label{rem:decomposition}
The above proposition can be approached from a different angle by trying to disintegrate the GNS-representation of a given state $\varphi$. This leads to a proof of decomposition~\eqref{eq:decomposition} that relies only on disintegration of representations rather than on Renault's disintegration theorem for groupoids~\cite{Ren}, which is a more delicate result, especially in the non-Hausdorff setting. Let us sketch this argument, which elaborates on the discussion in~\cite{N}*{Section~2}.

So, assume we have a state $\varphi$ on $C^{*}(\G)$ with $C_{0}(\Gu )$ in its centralizer. Let $(H,\pi,\xi)$ be the corresponding GNS-triple. By \cite{N}*{Lemma~2.2} and its proof, there is a representation $\rho\colon C_0(\Gu)\to\pi(C^*(\G))'\subset B(H)$ such that $\rho(h)\pi(a)\xi=\pi(ah)\xi$ for all $h\in C_0(\Gu)$ and $a\in C^*(\G)$. We can then disintegrate $\pi$ with respect to $\rho(C_0(\Gu))$, so that $\pi=\int^\bigoplus_{\Gu}\pi_x\,d\nu(x)$ for representations $\pi_x\colon C^*(\G)\to B(H_x)$ and $\rho$ becomes the representation by diagonal operators. If $\xi=(\xi_x)_x$, we define a probability measure $\mu$ on $\Gu$ by $d\mu(x)=\|\xi_x\|^2d\nu(x)$ and, for $\mu$-almost all $x\in\Gu$, states~$\psi_x$ on $C^*(\G)$ by $\psi_x=\|\xi_x\|^{-2}(\pi_x(\cdot)\xi_x,\xi_x)$. Then $\varphi=\int_\Gu\psi_x\,d\mu(x)$.

Since $\pi(h)\xi=\rho(h)\xi$, we have $\pi_x(h)\xi_x=h(x)\xi_x$, and hence $\psi_x(ah)=\psi_x(ha)=h(x)\psi_x(a)$, for $\mu$-a.e.~$x$ and all $h\in C_0(\Gu)$ and $a\in C^*(\G)$. By the analogue of~\eqref{eq:f0} for the full norm, which is proved using the same arguments, we conclude that, for $\mu$-a.e.~$x$ and all $f\in C_c(\G)$, the value $\psi_x(f)$ depends only on $\eta_x(f)$, so that $\psi_x=\varphi_x\circ\eta_x$ on $C_c(\G)$ for a linear functional $\varphi_x$ on $C_c(\Gxx)$. This linear functional is positive, hence it extends to a state on $C^*(\Gxx)$, since by Lemma~\ref{lem21}, for every $h\in C_c(\Gxx)$, we can find $f\in C_c(\G)$ such that $h^**h=\eta_x(f^**f)$.
\hfill$\diamondsuit$
\end{remark}

Proposition~\ref{prop32} can be extended to a class of weights on $C^*(\G)$. Let us first recall some terminology. A weight $\varphi$ on a C$^*$-algebra $A$ is a map $\varphi\colon  A_{+} \to [0, +\infty]$ satisfying $\varphi(a+b)=\varphi(a)+\varphi(b)$ and $\varphi(\lambda a)=\lambda \varphi(a)$ for all $a, b\in A_{+}$ and $\lambda\ge0$. A weight $\varphi$ is called densely defined when $\{a\in A_{+} \ | \ \varphi(a) < \infty\}$ is dense in $A_{+}$, and it is called lower semi-continuous if $\{ a\in A_{+} \ | \ \varphi(a)\leq \lambda \}$ is closed for all $\lambda \geq 0$. We call $\varphi$ \emph{proper} if $\varphi$ is nonzero, densely defined and lower semi-continuous. We use the standard notation
$$
\mathfrak{N}_{\varphi} := \{ a\in A \ | \ \varphi(a^{*}a) < \infty\}\qquad\text{and}\qquad \mathfrak{M}_{\varphi}^{+}: = \{a\in A_{+} \ | \ \varphi(a) < \infty \}.
$$
Any weight $\varphi$ extends uniquely to a linear functional on the subspace $\mathfrak{M}_{\varphi}:=\mathfrak{N}_{\varphi}^*\mathfrak{N}_{\varphi}=\text{span} \ \mathfrak{M}_{\varphi}^{+}$, which is a dense $*$-subalgebra of $A$ if $\varphi$ is densely defined.

Let $\{ \varphi_{x}\}_{x\in \Gu}$ be a $\mu$-measurable field of states for a non-zero Radon measure $\mu$ on $\Gu$. Then the map $x\mapsto (\varphi_x\circ\vartheta_x)(a)$  is $\mu$-measurable for all $a\in C^{*}(\G)_{+}$, and hence we can define a map $\varphi\colon  C^{*}(\G)_{+} \to [0, +\infty]$ by
\begin{equation} \label{eqweight}
\varphi(a): = \int_{\Gu} (\varphi_x\circ\vartheta_x)(a) \, d\mu(x) \quad \text{for}\quad a\in C^{*}(\G)_{+} .
\end{equation}
Since $\G$ can be covered by bisections and $\mu$ is Radon, it follows from Lemma \ref{lem:span} that $\varphi$ is finite on~$C_{c}(\G)_{+}$. By Fatou's lemma it is also lower semi-continuous. It follows that $\varphi$ is a proper weight on $C^{*}(\G)$, with $C_{c}(\G)\subset \mathfrak{M}_{\varphi}$. In conclusion we can associate a proper weight $\varphi$ on $C^{*}(\G)$ to any pair $( \mu  , \{ \varphi_{x}\}_{x\in \Gu})$ consisting of a Radon measure $\mu$ on $\Gu$ and a $\mu$-measurable field of states $\{ \varphi_{x}\}_{x\in \Gu}$. We say that the weight $\varphi$ factors through $C_{r}^{*}(\G)$ if there exists a proper weight $\tilde{\varphi}$ on~$C_{r}^{*}(\G)$ such that $\varphi(a)=\tilde{\varphi}(\pi(a))$ for all $a\in C^{*}(\G)_{+}$, where $\pi\colon C^*(\G)\to C^*_r(\G)$ is the canonical surjection.

\begin{prop} \label{propwft}
Let $\G$ be a second countable locally compact \'etale groupoid and $\varphi$ be the proper weight on $C^{*}(\G)$ associated to a pair $( \mu  , \{ \varphi_{x}\}_{x\in \Gu})$ consisting of a Radon measure $\mu$ and a $\mu$-measurable field of states $\{ \varphi_{x}\}_{x\in \Gu}$. Then the following conditions are equivalent:
\begin{enumerate}
\item \label{iw1}$\varphi$ factors through $C_{r}^{*}(\G)$;
\item \label{iw2}$\varphi((\ker\pi)_{+})=0$;
\item \label{iw3} $\varphi_x$ factors through $C^*_e(\Gxx)$ for $\mu$-almost all $x$.
\end{enumerate}
\end{prop}

\bp
That \eqref{iw1} implies \eqref{iw2} is clear. Assume now that \eqref{iw2} is true. Let $\{V_{n}\}^\infty_{n=1}$ be an increasing sequence of open sets in $\Gu$ with compact closures such that $\Gu = \cup^\infty_{n=1}V_{n}$ . Since $\mu$ is a Radon measure, the measures $\mu_{n}:=\mu|_{V_{n}}$ are all finite regular Borel measures on $\Gu$. Let $\psi_{n}$ be the positive linear functional on $C^*(\G)$ associated to $(\mu_{n}, \{\varphi_{x}\}_{x \in \Gu} )$. It then follows from \eqref{eqweight} that for any $a\in C^{*}(\G)_{+}$ we have
$$
\varphi(a) = \int_{\Gu} (\varphi_x\circ\vartheta_x)(a)\,d\mu(x) \geq  \int_{V_{n}} (\varphi_x\circ\vartheta_x)(a)\,d\mu(x) = \psi_{n}(a)  .
$$
By assumption this must imply that $\psi_{n} ((\ker\pi)_{+})=0$ for all $n\in \mathbb{N}$, and hence \eqref{iw3} is true by Proposition \ref{prop32}.

Assume now that \eqref{iw3} is true. Then similarly to $\varphi$ we can define a proper weight $\tilde\varphi$ on $C^*_r(\G)$ by
$$
\tilde\varphi(a): = \int_{\Gu} (\varphi_x\circ\vartheta_{x,e})(a) \, d\mu(x) \quad \text{for}\quad a\in C^{*}_r(\G)_{+},
$$
where we view $\varphi_x$ as a state on $C^*_e(\Gxx)$ whenever it makes sense, which is the case for $\mu$-a.e.~$x$ by assumption. Since $\varphi_x\circ\vartheta_x=\varphi_x\circ\vartheta_{x,e}\circ\pi$, we obviously have $\varphi(a)=\tilde{\varphi}(\pi(a))$ for all $a\in C^{*}(\G)_{+}$. This completes the proof.
\ep

\begin{remark}
There is no known analogue of decomposition~\eqref{eq:decomposition} for arbitrary proper weights with centralizer containing $C_0(\Gu)$. If one assumes that $\G$ is Hausdorff and restricts attention to KMS weights with respect to the time evolution defined by a real-valued $1$-cocycle on $\G$, then it follows from Theorem~6.3 in \cite{C} that every KMS weight is given by a pair $(\mu, \{\varphi_{x}\}_{x\in \Gu})$ (satisfying some extra conditions), and hence Proposition~\ref{propwft} can be used to describe all KMS weights on~$C^{*}_{r}(\G)$ in this setting.
\end{remark}

\section{Graded groupoids}\label{section5}

Given a locally compact \'etale groupoid $\G$ and a discrete group $\Gamma$, by a $\Gamma$-valued $1$-cocycle on~$\G$ we mean a continuous homomorphism $\Phi\colon\G\to\Gamma$. We then say that $\Phi$ defines a $\Gamma$-grading on~$\G$. In this section we prove several extensions of Proposition~\ref{prop:transformation-groupoids} of the form that if $\G$ is a graded groupoid satisfying some extra assumptions, then the norm $\lVert\cdot\rVert_e$ coincides with the reduced norm.

\smallskip

Let us first give a few examples of graded groupoids.

\begin{example}[Semidirect products]\label{ex:semidirect}
Assume a discrete group $\Gamma$ acts by automorphisms on a locally compact \'etale groupoid $\G$. Then the semidirect product $\Gamma\ltimes\G$ is the space $\Gamma\times\G$ with product
$$
(\gamma_2,g_2)(\gamma_1,g_1)=(\gamma_2\gamma_1,\gamma_1^{-1}(g_2)g_1).
$$
Therefore the unit space of $\Gamma\times\G$ is $\{e\}\times\Gu$, and if we identify it with $\Gu$, then the range and source maps are
$$
r(\gamma,g)=r(\gamma(g)),\qquad s(\gamma,g)=s(g).
$$
Equipped with the product topology, $\Gamma\ltimes\G$ becomes a locally compact \'etale groupoid. The map $\Phi\colon \Gamma\ltimes\G\to\Gamma$, $\Phi(\gamma,g)=\gamma$, defines a grading on $\Gamma\ltimes\G$.

Note that in general the isotropy groups $\Gamma\ltimes\G$ are not determined by those of $\G$ and the stabilizers of the $\Gamma$-action, but for every $x\in\Gu$ we at least have $(\Gamma\ltimes\G)^x_x\cap\ker\Phi=\{e\}\times\Gxx\cong\Gxx$, so that $(\Gamma\ltimes\G)^x_x$ is an extension of a subgroup of $\Gamma$ by $\Gxx$.
\hfill$\diamondsuit$
\end{example}

\begin{example}[Partial group actions]\label{ex:partial}
Given a discrete group $\Gamma$ and a locally compact Hausdorff space~$X$, a partial action of $\Gamma$ on $X$ is given by the following data. For every $\gamma\in\Gamma$, we are given open subsets $X_{\gamma^{-1}}$ and $X_\gamma$ of $X$ and a homeomorphism $X_{\gamma^{-1}}\to X_\gamma$, which we denote simply by~$\gamma$, such that $X_e=X$ and, for all $\gamma_1,\gamma_2\in\Gamma$, we have $(\gamma_1\gamma_2)x=\gamma_1(\gamma_2x)$ whenever the right hand side is well-defined (in other words, the homeomorphism $\gamma_1\gamma_2$ is an extension of the composition of~$\gamma_1$ and~$\gamma_2$).

Every partial action of $\Gamma$ defines a transformation-type groupoid $\G$~\cite{Ab2}:
$$
\G=\{(y,\gamma,x)\in X\times\Gamma\times X: y=\gamma x\},
$$
with the obvious product
\begin{equation}\label{eq:product}
(z,\gamma_2,y)(y,\gamma_1,x)=(z,\gamma_2\gamma_1,x).
\end{equation}
Equipped with the topology inherited from the product topology on $X\times\Gamma\times X$, the groupoid $\G$ becomes a locally compact Hausdorff \'etale groupoid. The map $\Phi\colon \G\to\Gamma$, $\Phi(y,\gamma,x)=\gamma$, defines a $\Gamma$-grading on $\G$. This grading is \emph{injective} in the sense that $\ker(\Phi|_{\Gxx})=\{x\}$ for all $x\in\Gu$.

\smallskip

The simplest way of getting a partial action of $\Gamma$ is to start with a genuine action on a locally compact Hausdorff space $Y$ containing $X$ as an open subset. We then get a partial action of~$\Gamma$ on~$X$ by restriction, so that $X_\gamma=X\cap\gamma(X)$. In this case the groupoid $\G$ of the partial action is simply the reduction of the transformation groupoid $\Gamma\times Y$ by $X$. But not all partial actions can be obtained this way, at least if we require $Y$ to be Hausdorff. A necessary condition for a partial action to arise as a reduction of a genuine action is that the graphs $\{(x,\gamma x): x\in X_{\gamma^{-1}}\}$ must be closed in $X\times X$ for all $\gamma\in\Gamma$. This condition is actually also sufficient, see~\cite{Ab}*{Proposition~2.10 and Remark 2.11}.

As an example, assume we are given partial homeomorphisms $S_1$ and $S_2$ on $X$, so that we are given open subsets $A_i$ and $B_i$ of $X$ ($i=1,2$) and homeomorphisms $S_i\colon A_i\to B_i$. From this we can construct a partial action of the free group $F_2$ with two generators $s_1$ and $s_2$. Namely, given a reduced nonempty word $\gamma=s_{i_1}^{a_1}\dots s_{i_n}^{a_n}\in F_2$, the action of $\gamma$ is given by $S_{i_1}^{a_1}\dots S_{i_n}^{a_n}$. If the graph of~$S_1$ or~$S_2$ is not closed in $X\times X$, this partial action is not a reduction of a genuine action.
\hfill$\diamondsuit$
\end{example}

\begin{example}[Generalized Deaconu--Renault groupoids]\label{ex:DR0}
We recall the construction of Exel and Renault~\cite{ER} generalizing the Deaconu--Renault groupoids of local homeomorphisms~\citelist{\cite{Rbook}\cite{D}}.

Let $\Gamma$ be a discrete group and $S\subset \Gamma$ be a submonoid such that $SS^{-1}\subset S^{-1}S$. Assume $S$ acts by local homeomorphisms on a locally compact Hausdorff space $X$. We define a groupoid by
\begin{equation}\label{eq:DR}
\G:=\{(y,\gamma,x)\in X\times \Gamma\times X:\exists\, s,t\in S\ \text{such that}\ \gamma=s^{-1}t\ \ \text{and}\ sy=tx\}.
\end{equation}
The product is defined by~\eqref{eq:product}. The sets
$$
\{(y,s^{-1}t,x):x\in A,\ y\in B, sy=tx\},
$$
where $A,B\subset X$ are open and $s,t\in S$, form a basis of topology on $\G$. This topology is stronger than, and in general different from, the one induced by the product topology on $X\times\Gamma\times X$. The map $\Phi\colon\G\to\Gamma$, $\Phi(y,\gamma,x)=\gamma$, defines an injective $\Gamma$-grading on~$\G$.

\smallskip

In most concrete examples of this construction studied in the literature the group $\Gamma$ is abelian, which is not particularly interesting for our purposes, so let us give a class of examples with potentially more complicated isotropy groups.

Assume $G$ is a discrete group, $\theta\in\operatorname{Aut}(G)$, and consider the group $\Gamma:=\Z\ltimes_\theta G$. Consider the submonoid  $S:=\Z_+\ltimes_\theta G\subset \Gamma$. Then $SS^{-1}=S^{-1}S=\Gamma$. Assume we have an action of $S$ on $X$ by local homeomorphisms. Then the corresponding groupoid is
\begin{equation}\label{eq:generalizedER}
\G=\{(y,k-l,g,x)\in X\times(\Z\ltimes_\theta G)\times X:k,l\ge0,\ \sigma^l(y)=\sigma^k(gx)\},
\end{equation}
where $\sigma$ is the local homeomorphism defined by the action of $1\in\Z_+$.

Note that to have an action of $S$ by local homeomorphisms is the same as having a local homeomorphism $\sigma\colon X\to X$ and an action of~$G$ on~$X$ by homeomorphisms such that $\sigma(gx)=\theta(g)\sigma(x)$ for all $g\in G$ and $x\in X$. A simple, but possibly not the most exciting, way of producing such examples is to consider a local homeomorphism $\sigma_1\colon X_1\to X_1$ and an action of $\Gamma$ on $X_2$, and then consider the local homeomorphism $\sigma=(\sigma_1,1)$ on $X:=X_1\times X_2$ together with the action of~$G$ on the second factor of $X$.
\hfill$\diamondsuit$
\end{example}

In some cases a graded groupoid can be transformed into a new injectively graded one.

\begin{example}\label{ex:DR}
Let $\G$ be a locally compact \'etale groupoid with a grading $\Psi\colon\G\to\Gamma$. Assume that the following condition is satisfied: if $\Psi(g)=e$ for some $g\in\Gxx$ and $x\in\Gu$, then $r=s$ in a neighbourhood of $g$. We then construct a new groupoid $\G_\Psi$ as follows.

Put $X:=\Gu$ and consider the set
$$
\G_\Psi:=\{(y,\gamma,x)\in X\times\Gamma\times X: \exists\, g\in\G^y_x\ \text{such that}\ \Psi(g)=\gamma\}.
$$
The product is again defined by~\eqref{eq:product}.

To define a topology on $\G_\Psi$, take a bisection $U$ of $\G$ and consider the set
$$
U_\Psi:=\{(r(g),\Psi(g),s(g)): g\in U\}.
$$
We claim that these sets form a basis of a topology on $\G_\Psi$. It is clear that the sets $U_\Psi$ cover $\G_\Psi$. Assume now that $U$ and $V$ are two bisections of $\G$ and $(y,\gamma,x)\in U_\Psi\cap V_\Psi$. Let $g\in U$ and $h\in V$ be the unique elements such that $g,h\in\G^y_x$. Then $g^{-1}h\in\Gxx$ and $\Psi(g^{-1}h)=e$. By our assumption, there is a bisection $O$ containing $g^{-1}h$ such that $r=s$ on $O$ and $\Psi(O)=\{e\}$. Note that $UO$ (the set of pairwise products) is a bisection and $gO=h$. Hence we can find an open neighbourhood $W$ of $g$ such that $W\subset U$, $s(W)\subset r(O)$ and $h\in WO\subset V$. Since $r(kO)=r(k)$, $s(kO)=s(k)$ and $\Psi(kO)=\Psi(k)$ for all $k\in W$, we then have $(y,\gamma,x)\in W_\Psi\subset U_\Psi\cap V_\Psi$.

Observe next that the open sets $U_\Psi$ are bisections. It is then easy to see that $\G_\Psi$ is a locally compact Hausdorff \'etale groupoid. The map $\Phi\colon\G_\Psi\to\Gamma$, $\Phi(y,\gamma,x)=\gamma$, defines an injective $\Gamma$-grading on~$\G_\Psi$. For the isotropy groups we have $(\G_\Psi)^x_x\cong\Psi(\Gxx)$.

\smallskip

The above construction can be viewed as a generalization of Example~\ref{ex:DR0}. Indeed, consider an action of $S\subset\Gamma$ on $X$ by local homeomorphisms as in that example. For every $s\in S$, $s\ne e$, we can cover $X$ by open sets $U_{s,i}$ such that $s|_{U_{s,i}}$ is a homeomorphism onto $sU_{s,i}$. Consider the free group~$F$ with generators $g_{s,i}$ for all $s$ and $i$. We can define a partial action of $F$ on $X$ similarly to Example~\ref{ex:partial}, with $g_{s,i}$ acting by $s|_{U_{s,i}}$. Let $\G$ be the corresponding groupoid. It is $F$-graded, but we rather view it as $\Gamma$-graded using the homomorphism $\phi\colon F\to\Gamma$ such that $\phi(g_{s,i})=s$ for all $s$ and $i$.

To check that our assumption on the corresponding cocycle $\Psi\colon\G\to\Gamma$ is satisfied, assume $g_{s_1,i_1}^{a_1}\dots g_{s_n,i_n}^{a_n}x=x$ for some $x\in X$ and $g_{s_1,i_1}^{a_1}\dots g_{s_n,i_n}^{a_n}\in F$ ($a_j=\pm1$) such that $s_{1}^{a_1}\dots s_{n}^{a_n}=e$. We have to show that $g_{s_1,i_1}^{a_1}\dots g_{s_n,i_n}^{a_n}y=y$ for all $y$ close to $x$. By the assumption $SS^{-1}\subset S^{-1}S$, we can write $st^{-1}$ for any given $s,t\in S$ as $p^{-1}q$, which implies that for any indices $i,j$ and any point $y\in X$ we can find indices $k,l$ such that $g_{s,i}g_{t,j}^{-1}z=g_{p,k}^{-1}g_{q,l}z$ for all $z$ close to $y$. Using this property repeatedly, we can move all the negative powers to the left, that is, without loss of generality we may assume that $a_1=\dots=a_m=-1$ and $a_{m+1}=\dots=a_n=1$ for some~$m$. But then for
$$
s:=s_{m+1}\dots s_n=s_m\dots s_1
$$
we have $sg_{s_1,i_1}^{a_1}\dots g_{s_n,i_n}^{a_n}y=sy$, and since $s$ acts by a local homeomorphism, we conclude that $g_{s_1,i_1}^{a_1}\dots g_{s_n,i_n}^{a_n}y=y$ for all $y$ close to $x$. It is now straightforward to check that the groupoid $\G_\Psi$, with its topology, is exactly the groupoid~\eqref{eq:DR}.

\smallskip

Consider a related example inspired by \cite{HL}*{Example~3.3}. Let $\G$ be the groupoid~\eqref{eq:generalizedER} defined by an action of $\Z_+\ltimes_\theta G$ on $X$ by local homeomorphisms. Take any discrete group $\Gamma$ together with a homomorphism $\phi\colon G\to\Gamma$ such that
$$
\Omega:=\{\gamma\in\Gamma: \phi(g)\gamma=\gamma\phi(\theta(g))\ \text{for all}\ g\in G\}\ne\emptyset.
$$
For any continuous $G$-invariant function $d\colon X\to\Omega$, we can then define a $1$-cocycle $\Psi\colon \G\to\Gamma$ by
$$
\Psi(y,k-l,g,x):=d(y)d(\sigma(y))\dots d(\sigma^{l-1}(y))d(\sigma^{k-1}(x))^{-1}\dots d(\sigma(x))^{-1}d(x)^{-1}\phi(g),
$$
if $\sigma^l(y)=\sigma^k(gx)$. Our condition on $\Psi$ reads as follows: if $\sigma^l(x)=\sigma^k(gx)$ and $\Psi(x,k-l,g,x)=e$, then $\sigma^l(y)=\sigma^k(gy)$ for all $y$ close to $x$. If it is satisfied, we get an injectively $\Gamma$-graded locally compact Hausdorff \'etale groupoid $\G_\Psi$.
For $G=\{e\}$ this groupoid coincides with the one defined in \cite{HL}*{Example~3.3}.
\hfill$\diamondsuit$
\end{example}

\begin{thm}\label{thm:e=r}
Assume $\Gamma$ is a discrete group and $\G$ is a $\Gamma$-graded locally compact \'etale groupoid, with grading $\Phi\colon\G\to\Gamma$. For a fixed $x\in\Gu$, assume the group $\ker(\Phi|_{\Gxx})$ is amenable and $\Phi(\Gxx)$ is exact. Then $\lVert\cdot\rVert_e=\lVert\cdot\rVert_r$ on $\C\Gxx$.
\end{thm}

We need some preparation to prove this theorem. The $1$-cocycle $\Phi$ defines a $\Gamma$-grading on $C_c(\G)$, or equivalently, a coaction $\delta\colon C_c(\G)\to C_c(\G)\odot\C \Gamma$ of the Hopf algebra $(\C \Gamma,\Delta)$, where $\Delta(\lambda_g)=\lambda_g\otimes\lambda_g$ and $\odot$ denotes the algebraic tensor product. Namely, if we view the elements of $C_c(\G)\odot\C \Gamma$ as $\C \Gamma$-valued functions on $\G$, then
$$
\delta(f)(g)=f(g)\lambda_{\Phi(g)}.
$$

The following lemma and its proof are standard.

\begin{lemma}
The map $\delta$ extends to an injective $*$-homomorphism $C^*_r(\G)\to C^*_r(\G)\otimes C^*_r(\Gamma)$, which defines a coaction of $(C^*_r(\Gamma),\Delta)$ on $C^*_r(\G)$.
\end{lemma}

\bp
For every $y\in\Gu$, define a unitary operator
$$
W_y\colon \ell^2(\G_y)\otimes \ell^2(\Gamma)\to \ell^2(\G_y)\otimes \ell^2(\Gamma)\quad\text{by}\quad (W_y\xi)(g,\gamma):=\xi(g,\Phi(g)^{-1}\gamma).
$$
A simple computation shows then that for every $f\in C_c(\G)$ we have
$$
W_y(\rho_y(f)\otimes1)=(\rho_y\otimes\iota)\big(\delta(f)\big)W_y.
$$
It follows that we have an injective $*$-homomorphism $\delta_y\colon \rho_y(C^*_r(\G))\to \rho_y(C^*_r(\G))\otimes C^*_r(\Gamma)$ defined by
$\delta_y(a)=W_y(a\otimes1)W_y^*$, and then $(\rho_y\otimes\iota)\circ\delta=\delta_y\circ\rho_y$. Taking the direct sum of the representations~$\rho_y$ we conclude that the homomorphisms $\delta_y$ define the required extension of $\delta$.
\ep

\begin{lemma}\label{lem:e-norm_tensor}
Let $A$ be a C$^*$-algebra and $\A\subset A$ a dense $*$-subalgebra. For $h\in C_c(\Gxx)\odot \A$, define
$$
\|h\|_e=\inf\{\|f\|_r: f\in C_c(\G)\odot \A,\ \ (\eta_x\otimes\iota)(f)=h\},
$$
where $\|f\|_r$ is the norm of $f$ in the minimal tensor product $C^*_r(\G)\otimes A$. Then $\lVert\cdot\rVert_e$ extends uniquely to a C$^*$-cross norm on $C^*_e(\Gxx)\odot A$, and if $(q_V)_V$ is a net as in Theorem~~\ref{thmnorm}, then
\begin{equation}\label{eq:e-norm-limit2}
\|h\|_e=\lim_{V}\|(q_V\otimes1)f(q_V\otimes1)\|_r
\end{equation}
for every $h\in C_c(\Gxx)\odot \A$ and any $f\in C_c(\G)\odot \A$ such that $(\eta_x\otimes\iota)(f)=h$.
\end{lemma}

We will need this lemma only for $\A=\C G\subset A=C^*_r(G)$ for a discrete group $G$, in which case a large part of it is a consequence of Theorem~\ref{thmnorm} applied to the groupoid $\G\times G$. The general case is, however, of some independent interest and the proof is not much longer.

\bp
Identity~\eqref{eq:e-norm-limit2} and the C$^*$-seminorm property of $\lVert\cdot\rVert_e$ are proved exactly as Lemmas~\ref{lem:e-expression} and~\ref{lem:c-star-norm}, if we view the elements of $C_c(\G)\odot \A$ as $\A$-valued functions on $\G$. Furthermore, identities~\eqref{eq:e-norm-limit} and~\eqref{eq:e-norm-limit2} imply that $\|h\otimes a\|_e=\|h\|_e\|a\|$ for all $h\in C_c(\Gxx)$ and $a\in \A$.

Next, the seminorm $\lVert\cdot\rVert_e$ on $ C_c(\Gxx)\odot \A$ dominates the norm defined by the embedding $ C_c(\Gxx)\odot \A\hookrightarrow C^*_e(\Gxx)\otimes A$.
This is proved similarly to the proof of the inequality $\lVert\cdot\rVert_e'\ge\lVert\cdot\rVert_e$ in Lemma~\ref{lem:c-star-norm}: since $\vartheta_{x,e}\otimes\iota\colon C^*_r(\G)\otimes A\to C^*_e(\Gxx)\otimes A$ is a contraction, for all $h\in C_c(\Gxx)\odot \A$ and $f\in C_c(\G)\odot \A$ such that $(\eta_x\otimes\iota)(f)=h$, the norm of $h$
in $C^*_e(\Gxx)\otimes A$ is not larger than the norm of $f$ in $C^*_r(\G)\otimes A$. In particular, $\lVert\cdot\rVert_e$ is a C$^*$-norm on $ C_c(\Gxx)\odot \A$, and if we denote by $B$ the $\lVert\cdot\rVert_e$-completion of $ C_c(\Gxx)\odot \A$, then the identity map on $C_c(\Gxx)\odot \A$ extends to a $*$-homomorphism $\varphi\colon B\to C^*_e(\Gxx)\otimes A$.

On the other hand, the equality $\|h\otimes a\|_e=\|h\|_e\|a\|$ implies that the identity map on $ C_c(\Gxx)\odot \A$ extends to a $*$-homomorphism $\psi\colon C^*_e(\Gxx)\odot A\to B$. By construction, $\varphi\circ\psi$ is the identity map on $C^*_e(\Gxx)\odot A$, hence $\psi$ is injective. Therefore $C^*_e(\Gxx)\odot A$ can be viewed as a subalgebra of $B$, so the restriction of the norm on $B$ to this subalgebra gives the required extension of $\lVert\cdot\rVert_e$.

Finally, $C_c(\Gxx)\odot \A$ is dense in $C^*_e(\Gxx)\odot A$ with respect to any C$^*$-cross norm on $C^*_e(\Gxx)\odot A$, so any such norm is completely determined by its restriction to $C_c(\Gxx)\odot \A$.
\ep

Therefore the completion of $C_c(\Gxx)\odot \A$ with respect to $\lVert\cdot\rVert_e$ is a C$^*$-tensor product of $C^*_e(\Gxx)$ and~$A$.

\begin{lemma}\label{lem:exact-completion}
 If $A$ is an exact C$^*$-algebra, then the completion of $\C\Gxx\odot \A$ with respect to $\lVert\cdot\rVert_e$ is the minimal tensor product $C^*_e(\Gxx)\otimes A$.
\end{lemma}

\bp
The lemma is obviously true for nuclear C$^*$-algebras, in particular, for finite dimensional ones. Assuming that $A$ is only exact, represent it as a concrete C$^*$-algebra $A\subset B(H)$. Then the embedding map $A\to B(H)$ is nuclear. It follows that if we fix $f\in C_c(\G)\odot \A$ and $\eps>0$, then we can find contractive completely positive maps
$\theta\colon A\to\Mat_n(\C)$ and $\psi\colon\Mat_n(\C)\to B(H)$ such that $\|f-(\iota\otimes\psi\circ\theta)(f)\|_r<\eps$.

Put $h:=(\eta_x\otimes\iota)(f)\in \C\Gxx\odot \A$. Let us denote by $\lVert\cdot\rVert$ the minimal norm on $C^*_e(\Gxx)\odot B$ for $B=A$ and $B=\Mat_n(\C)$. Since the lemma is true for $\Mat_n(\C)$, by using \eqref{eq:e-norm-limit2} we get
\begin{align*}
\|h\|&\ge \|(\iota\otimes\theta)(h)\|=\lim_{V}\|(\iota\otimes\theta)\big((q_V\otimes1)f(q_V\otimes1)\big)\|_r\\
&\ge \lim_{V}\|(\iota\otimes\psi\circ\theta)\big((q_V\otimes1)f(q_V\otimes1)\big)\|_r\\
&\ge \lim_{V}\|(q_V\otimes1)f(q_V\otimes1)\|_r-\eps=\|h\|_e-\eps.
\end{align*}
As $\eps$ was arbitrary, we thus have $\|h\|\ge\|h\|_e$. Since the opposite inequality holds by the previous lemma, we get the result.
\ep

\bp[Proof of Theorem~\ref{thm:e=r}]
Since $\delta\colon C^*_r(\G)\to C^*_r(\G)\otimes C^*_r(\Gamma)$ is isometric and
$$
\delta(q_V*f*q_V)=(q_V\otimes1)\delta(f)(q_V\otimes1)
$$
for every $f\in C_c(\G)$, identities~\eqref{eq:e-norm-limit} and~\eqref{eq:e-norm-limit2} imply that $\delta$ induces an injective $*$-homomorphism~$\Delta_x$ from $C^*_e(\Gxx)$ into the $\lVert\cdot\rVert_e$-completion of $\C\Gxx\odot C^*_r(\Gamma)$. Put $G:=\Phi(\Gxx)\subset \Gamma$. Then the image of $\Delta_x$ is contained in the closure of $\C\Gxx\odot C^*_r(G)$ in the $\lVert\cdot\rVert_e$-completion of $\C\Gxx\odot C^*_r(\Gamma)$, hence (by~\eqref{eq:e-norm-limit2}) in  the $\lVert\cdot\rVert_e$-completion of $\C\Gxx\odot C^*_r(G)$. Since $G$ is exact by assumption, by Lemma~\ref{lem:exact-completion} we conclude that $\Delta_x$ can be viewed as a homomorphism $C^*_e(\Gxx)\to C^*_e(\Gxx)\otimes C^*_r(G)$.

For a discrete group $H$ and a subgroup $H_0$, denote by $\lambda_{H/H_0}$ the quasi-regular representation of~$H$ on $\ell^2(H/H_0)$ and by $\eps_H$ the trivial representation of $H$. Let us also denote weak containment and quasi-equivalence of representations by $\prec$ and $\sim$, respectively. Now, put $H:=\Gxx\cap\ker\Phi$ and let $\rho$ be a unitary representation of $\Gxx$ that integrates to a faithful representation of $C^*_e(\Gxx)$. Let $\rho\times \lambda_G$ denote the injective representation of $C^*_e(\Gxx)\otimes C^*_r(G)$ arising from taking the tensor product of the integrated representations of $\rho$ and $\lambda_{G}$, and let $\rho\otimes (\lambda_{G} \circ\Phi|_{\Gxx})$ denote the tensor product representation of $\Gxx$. Then $\big(\rho\otimes (\lambda_{G} \circ\Phi|_{\Gxx} )\big) (h) = (\rho \times \lambda_{G}) (\Delta_{x}(h))$ for $h\in C_{c}(\Gxx)$, so the injectivity of $\Delta_x$ implies that $\rho \prec \rho\otimes (\lambda_G\circ\Phi|_{\Gxx})$, hence
$$
\rho\prec\rho\otimes (\lambda_G\circ\Phi|_{\Gxx})\sim \rho\otimes \lambda_{\Gxx/H}
\sim\rho\otimes(\Ind\eps_H).
$$
Since $H$ is amenable by assumption, we have $\eps_H\prec\lambda_H$, hence
$$
\rho\otimes(\Ind\eps_H)\prec\rho\otimes(\Ind\lambda_H)\sim \rho\otimes\lambda_{\Gxx}\sim\lambda_{\Gxx},
$$
where in the last step we used Fell's absorption principle. Therefore $\rho\prec\lambda_{\Gxx}$. Since $\lVert\cdot\rVert_e$ dominates the reduced norm, we conclude that $\lVert\cdot\rVert_e=\lVert\cdot\rVert_r$.
\ep

\begin{remark}
If $G$ is an exact group, then Lemma~\ref{lem:exact-completion} for $A=C^*_r(G)$ can be rephrased by saying that for any locally compact \'etale groupoid $\G$ and any $x\in\Gu$ we have $C^*_e((\G\times G)^x_x)=C^*_e(\Gxx)\otimes C^*_r(G)$. If we could prove this for arbitrary $G$, the assumption of exactness in Theorem~\ref{thm:e=r} would be unnecessary.
\hfill$\diamondsuit$
\end{remark}

Applying Theorem~\ref{thm:e=r} to groupoids from Example~\ref{ex:semidirect} and recalling that the subgroups of exact groups are exact, we get the following.

\begin{cor}
Assume an exact discrete group $\Gamma$ acts on a locally compact \'etale groupoid $\G$. For a fixed $x\in\Gu$, assume that the isotropy group $\Gxx$ is amenable. Then $\lVert\cdot\rVert_e=\lVert\cdot\rVert_r$ on the group algebra of $(\Gamma\ltimes\G)^x_x$.
\end{cor}

Note that if $\G$ is topologically amenable (see \cite{ADR}*{Section 2.2.b}), then its isotropy groups are amenable.

\smallskip

Theorem~\ref{thm:e=r} also applies to the injectively graded groupoids from Examples~\ref{ex:partial}--\ref{ex:DR}, when the grading group $\Gamma$ is exact. Although the assumption of exactness is very mild, it is still somewhat unsatisfactory, since Theorem~\ref{thm:e=r} does not fully cover even the trivial case of transformation groupoids (Proposition~\ref{prop:transformation-groupoids}). We will therefore prove the following result, which assumes much more about the groupoid structure, but does not need exactness of the grading group.

\begin{thm} \label{tmred}
Suppose $\Gamma$ is a discrete group and $\G$ is a $\Gamma$-graded locally compact \'etale groupoid, with grading $\Phi\colon\G\to\Gamma$. For a fixed $x\in \Gu $, assume $\Phi$ is injective on $\Gxx$  and there is a family $\{U_{g}\}_{g\in \G_{x}^{x}}$ of bisections such that:
\begin{enumerate}
\item $g\in U_{g}$ and $\Phi(U_{g})=\Phi(g)$ for all $g\in\Gxx$;
\item $U_{x}=\Gu$ and, for all $g\in \G_{x}^{x}$, $U_{g^{-1}}=U_{g}^{-1}$;
\item if $g_{1}, \dots , g_{n} \in \G_{x}^{x}$ and $g_{1}\dots g_{n}=x$, then $U_{g_{1}} U_{g_{2}} \dots U_{g_{n}}\subset\Gu$.
\end{enumerate}

Then $\lVert \cdot \rVert_{e} =\lVert \cdot \rVert_{r}$ on $\C\Gxx$.
\end{thm}

Again, we need some preparation to prove the theorem. Take $y\in\Gu$.

\begin{lemma}
Define a binary relation $\sim_{x}$ on $\G_{y}$ by
\begin{equation*}
g \sim_{x} h \iff \exists\ g_{1}, \dots , g_{n}\in \G_{x}^{x} \text{ with } h= U_{g_{1}}\dots U_{g_{n}} g .
\end{equation*}
Then $\sim_{x}$ is an equivalence relation.
\end{lemma}

\begin{proof}
Since $U_{x}=\Gu $, for every $g \in \G_{y}$ we have that $g = U_{x}g$, and hence $g \sim_{x} g$.

Assume now that $g \sim_{x} h$, and let $g_{1}, \dots , g_{n}\in \G_{x}^{x}$ with $h= U_{g_{1}}\dots U_{g_{n}} g$. Then $r(h)\in r(U_{g_{1}})=s(U_{g_{1}^{-1}})$ and hence the composition
$U_{g_{1}^{-1}}h= U_{g_{1}^{-1}}U_{g_{1}}\dots U_{g_{n}} g$ gives an element in $\G_{y}$. By our assumptions, $U_{g_{1}^{-1}}U_{g_{1}}$ is contained in $\Gu $, so
\begin{equation*}
 U_{g_{1}^{-1}}U_{g_{1}}U_{g_{2}}\dots U_{g_{n}} g =U_{g_{2}}\dots U_{g_{n}} g .
\end{equation*}
Continuing like this we get that $U_{g_{n}^{-1}}\dots U_{g_{2}^{-1}} U_{g_{1}^{-1}}h= g$, and hence $h \sim_{x} g$.

To complete the proof that $\sim_{x}$ is an equivalence relation assume that $g \sim_{x} h$ and $h \sim_{x} k$. There exist $g_{1}, \dots , g_{n}\in \G_{x}^{x}$ and $h_{1}, \dots , h_{m}\in \G_{x}^{x}$ such that $ h= U_{g_{1}}\dots U_{g_{n}} g$ and $k = U_{h_{1}}\dots U_{h_{m}} h$. Combining these two identities we get that
\begin{equation*}
 k = U_{h_{1}}\dots U_{h_{m}} h = U_{h_{1}}\dots U_{h_{m}}  U_{g_{1}}\dots U_{g_{n}} g
\end{equation*}
and hence $g \sim_{x} k$. %In conclusion $\sim_{x}$ is an equivalence relation.
\end{proof}

Let $K_y:=\G_{y} /{\sim_{x}}$. For each $\kappa \in K_y$, set
\begin{equation*}
H_{\kappa}:=\overline{\text{span} \{ \delta_{g} : g \in \kappa  \}}\subset\ell^2(\G_y).
\end{equation*}
Then $\ell^{2}(\G_{y}) = \bigoplus_{\kappa\in K_y} H_{\kappa}$.

\begin{lemma} \label{ldec}
If $f\in C_{c}(\G)$ is zero outside $\bigcup_{g\in \G_{x}^{x}} U_{g}$, then each subspace $H_{\kappa}$ ($\kappa\in K_y$) is invariant under the action of $\rho_{y}(f)$.
\end{lemma}

\begin{proof}
Take $g\in\kappa$. Since
\begin{equation*}
\rho_{y}(f) \delta_{g} =\sum_{h \in \G_{r(g)}} f(h) \delta_{h  g},
\end{equation*}
it suffices to prove that $h g \in \kappa$ whenever $f(h)\neq 0$. If $f(h) \neq 0$, then by assumption $h \in U_{k}$ for some $k\in \G_{x}^{x}$. This means that
$hg = U_{k}g$, and hence $hg \sim_{x} g$. This proves the lemma.
\end{proof}

The following lemma will allow us to embed $H_\kappa$ into $\ell^2(\Gamma)$.

\begin{lemma} \label{lkappa}
For every $\kappa \in K_y$, the map $\kappa\ni g \mapsto \Phi(g)\in \Gamma$ is injective.
\end{lemma}

\begin{proof}
Assume $g, h \in \kappa$ are such that $\Phi(g)=\Phi(h)$. By definition, there exist $g_{1}, \dots , g_{n} \in \G_{x}^{x}$ such that $g=U_{g_{1}}U_{g_{2}}\dots U_{g_{n}} h$. Applying $\Phi$ to both sides we get that
\begin{equation*}
\Phi(g)=\Phi(g_{1})\dots \Phi(g_{n})\Phi(h).
\end{equation*}
Since $\Phi(g)=\Phi(h)$ and $\Phi$ is injective on $\G_{x}^{x}$, this implies that $g_{1}\dots g_{n} = x$. By our assumptions we then have that $U_{g_{1}}U_{g_{2}}\dots U_{g_{n}}\subset \Gu$, and hence $g=h$.
\end{proof}

\begin{proof}[Proof of Theorem~\ref{tmred}]
Fix a nonzero function $h\in C_c(\Gxx)$. Let $\{g_{1}, \dots , g_{n}\} \subset\G_{x}^{x}$ be the support of~$h$. Choose an open set $V\subset\Gu$ such that $x\in V\subset\bigcap^n_{i=1}r(U_{g_i})$, the closure of $W_i:=r^{-1}(V)\cap U_{g_i}$ is compact and contained in $U_{g_i}$ for each $i$, and such that $\Phi(W_{i})=\Phi(g_{i})$ for each $i$, which in particular implies that the sets $W_1,\dots,W_n$ are disjoint.
Next, choose functions $f_{i}\in C_{c}(U_{g_i})$ such that $f_{i}(g)=h(g_{i})$ for all $g\in W_i$. Finally, choose a function $q\in C_{c}(\Gu )$ with $\supp q \subset V$ such that $0\leq q \leq 1$ and $q(x)=1$. We now define
\begin{equation*}
f:=\sum_{i=1}^{n} q*f_{i} \in C_{c}(\G) .
\end{equation*}
Then $\eta_x(f)=h$, and in order to prove the theorem it suffices to show that $\|\rho_y(f)\|\le\|h\|_r$ for all $y\in\Gu$.

By Lemma~\ref{ldec} it is then enough to show that $\|\rho_y(f)|_{H_\kappa}\|\le\|h\|_r$ for all $y\in\Gu$ and $\kappa\in K_y=\G_y/{\sim_x}$. This can be reformulated as follows. By identifying $\Gxx$ with the subgroup $\Phi(\Gxx)$ of $\Gamma$, we can extend the function $h$ by zero to a function $\tilde h$ on $\Gamma$, so that $\tilde h(\Phi(g))=h(g)$ for $g\in\Gxx$ and $\tilde h=0$ on $\Gamma\setminus\Phi(\Gxx)$. Denote by $\lambda_\Gamma$ the regular representation of $\Gamma$. Since $\lambda_\Gamma\circ\Phi|_{\Gxx}$ decomposes into a direct sum of copies of the regular representation of $\Gxx$, we have $\|h\|_r=\|\lambda_\Gamma(\tilde h)\|$. Therefore we need to show that
\begin{equation}\label{eq:rhoz-r}
\|\rho_y(f)|_{H_\kappa}\|\le\|\lambda_\Gamma(\tilde h)\|.
\end{equation}

Let us show first that
\begin{equation}\label{eq:f-h}
f(gk^{-1})= q(r(g))\tilde h (\Phi(gk^{-1})  ) \quad\text{for all}\quad g,k\in\kappa.
\end{equation}
If $gk^{-1}\in\bigcup^n_{i=1} W_i$, then this follows by the definition of $f$ and $\tilde h$. If $gk^{-1}\in (\bigcup^n_{i=1} U_{g_i})\setminus (\bigcup^n_{i=1} W_i)$, then both sides of~\eqref{eq:f-h} are zero, again by the definition of $f$ and since $q(r(g))=0$. Assume now that $gk^{-1}\notin\bigcup^n_{i=1} U_{g_i}$.
Then $f(gk^{-1})=0$, and we are done if $r(g)\notin V$. So assume that $r(g)\in V$. If $\Phi(gk^{-1})\notin \{\Phi(g_{i})\}_{i=1}^{n}$ then the right-hand side of~\eqref{eq:f-h} is zero, so let us assume that $\Phi(gk^{-1})=\Phi(g_{l})$ for some $l\in \{1, \dots, n\}$. Since $g,k \in \kappa$, there are elements $s_{1}, \dots, s_{m}\in\Gxx$ such that $g=U_{s_{1}}\dots U_{s_{m}}k$. Hence $g = r_{1} \dots r_{m} k$ for uniquely defined elements $r_i\in U_{s_i}$. Since $r(gk^{-1})\in V\subset r(U_{g_l})$, we have
\begin{equation*}
\emptyset\neq U_{g_{l}}^{-1}gk^{-1} = U_{g_{l}}^{-1}r_{1} \dots r_{m} \subset U_{g_{l}}^{-1}U_{s_{1}}\dots U_{s_{m}}.
\end{equation*}
But since $\Phi(s_{1} \dots s_{m})=\Phi(gk^{-1})=\Phi(g_{l})$, the last set is contained in $\Gu$ by assumption, hence $gk^{-1} \in U_{g_{l}}$, which is a contradiction. Thus \eqref{eq:f-h} is proved.

Now, by Lemma~\ref{lkappa} we have an isometry
$$
u\colon H_\kappa\to\ell^2(\Gamma),\quad u\delta_g:=\delta_{\Phi(g)}.
$$
In terms of this isometry, identity \eqref{eq:f-h} can be written as
$$
(\rho_y(f)\delta_k,\delta_g)=q(r(g))(\lambda_\Gamma(\tilde h)u\delta_k,u\delta_g)\quad\text{for all}\quad g,k\in\kappa.
$$
In other words,
$$
\rho_y(f)|_{H_\kappa}=u^*m_{\tilde q}\lambda_\Gamma(\tilde h)u,
$$
where $m_{\tilde q}\colon\ell^2(\Gamma)\to\ell^2(\Gamma)$ is the operator of multiplication by the function $\tilde q\colon\Gamma\to[0,1]$ defined by $\tilde q(\Phi(g)):=q(r(g))$ for $g\in\kappa$ and $\tilde q:=0$ on $\Gamma\setminus\Phi(\kappa)$. This clearly implies~\eqref{eq:rhoz-r}.
\end{proof}

\begin{cor}
Let $\G$ be the \'etale groupoid associated with a partial action of a discrete group~$\Gamma$ on a locally compact Hausdorff space $X$. Then, for every $x\in\Gu=X$, we have $\lVert\cdot\rVert_e=\lVert\cdot\rVert_r$ on~$\C\Gxx$.
\end{cor}

\bp
Using the notation from Example~\ref{ex:partial}, for every $g=(x,\gamma,x)\in\Gxx$ consider the bisection
$$
U_g:=\{(\gamma y,\gamma,y): y\in X_{\gamma^{-1}}\}.
$$
Then the assumptions of Theorem~\ref{tmred} are satisfied, so we get the result.
\ep

Note that if the assumptions of Theorem~\ref{thm:e=r} or~\ref{tmred} are satisfied for some $x\in\Gu$, then they are also satisfied for the points in the orbit of $x$. More generally, we have the following simple observation.

\begin{prop}\label{prop:orbit}
Assume $\G$ is a locally compact \'etale groupoid and $x\in\Gu$ is a point such that $\lVert\cdot\rVert_e=\lVert\cdot\rVert_r$ on $\C\Gxx$. Then $\lVert\cdot\rVert_e=\lVert\cdot\rVert_r$ on $\C\G^y_y$ for all points $y\in\Gu$ in the $\G$-orbit of $x$.
\end{prop}

\bp
If $g\in\G^y_x$, then the conjugation by $g$ defines an isomorphism $\pi\colon\Gxx\to\G^y_y$. The map $\rho\mapsto\rho\circ\pi$ defines a bijection between the sets $\mathcal S_y$ and $\mathcal S_x$ from Definition~\ref{def}, since $\Ind(\rho\circ\pi)\sim\Ind\rho$. Hence~$\pi$ extends to an isomorphism $C^*_e(\Gxx)\cong C^*_e(\G^y_y)$. Since it also extends to an isomorphism $C^*_r(\Gxx)\cong C^*_r(\G^y_y)$, this proves the result.
\ep

Following \cite{KhSk} we say that locally compact \'etale groupoids $\mathcal{H}$ and $\G$ are Morita equivalent if there exists a locally compact \'etale groupoid $\TT$ with unit space equal to the topological disjoint union $\TT^{(0)} = \Gu \sqcup \mathcal{H}^{(0)}$ such that $\mathcal{T}$ restricted to $\Gu$ and $\mathcal{H}^{(0)}$ is, resp., $\G$ and $\mathcal{H}$, and such that both $\Gu$ and $\mathcal{H}^{(0)}$ meet all $\mathcal{T}$-orbits.

\begin{cor}
If $\G$ and $\mathcal H$ are Morita equivalent locally compact \'etale groupoids and $\lVert\cdot\rVert_e=\lVert\cdot\rVert_r$ on~$\C\Gxx$ for all $x\in \Gu$, then the groupoid $\mathcal H$ has the same property.
\end{cor}

\bp
Let $\TT$ be a groupoid defining a Morita equivalence between $\G$ and $\mathcal H$. We can view every function $f\in C_c(\G)$ as an element of $C_c(\TT)$. For every $x\in\Gu$, we have $\rho^\TT_x(f)=\rho^\G_x(f)\oplus0$  with respect to the decomposition $\ell^2(\TT_x)=\ell^2(\G_x)\oplus\ell^2(\TT_x\setminus\G_x)$, where $\rho^\TT_x$ and $\rho^\G_x$ are the representations~\eqref{eq:rhox} defined for $\TT$ and $\G$, resp. We also have $\lVert \rho^\TT_{x}(f) \rVert= \lVert \rho^\TT_{y}(f) \rVert$ for all $x,y \in \TT^{(0)}$ belonging to the same $\TT$-orbit. This implies that the embeddings $C_{c}(\G) \hookrightarrow C_{c}(\TT)$ and $C_{c}(\mathcal{H}) \hookrightarrow C_{c}(\TT)$ are isometric with respect to the reduced norm. By an application of Theorem \ref{thmnorm} then the norms $\lVert\cdot\rVert_e$ agree on $\C\Gxx$ and $\C\TT_{x}^{x}$ for $x\in \Gu$ and on $\C\mathcal{H}_{y}^{y}$ and $\C\TT_{y}^{y}$ for $y\in \mathcal{H}^{(0)}$. Hence the corollary follows from Proposition~\ref{prop:orbit}.
\ep

As an application, using this corollary and Proposition \ref{prop:transformation-groupoids}, we see that if $\G$ is injectively graded, with grading $\Phi\colon\G\to\Gamma$ that is closed and transverse in the sense of~\cite{KhSk}*{Definition~1.6}, then $\lVert\cdot\rVert_e= \lVert\cdot\rVert_r$ on $\C\Gxx$ for all $x\in \Gu$, since by~\cite{KhSk}*{Theorem~1.8} the groupoid $\G$ is Morita equivalent to a transformation groupoid.

\bigskip


\begin{bibdiv}
\begin{biblist}

\bib{Ab}{article}{
   author={Abadie, Fernando},
   title={Enveloping actions and Takai duality for partial actions},
   journal={J. Funct. Anal.},
   volume={197},
   date={2003},
   number={1},
   pages={14--67},
   issn={0022-1236},
   review={\MR{1957674}},
   doi={10.1016/S0022-1236(02)00032-0},
}

\bib{Ab2}{article}{
   author={Abadie, Fernando},
   title={On partial actions and groupoids},
   journal={Proc. Amer. Math. Soc.},
   volume={132},
   date={2004},
   number={4},
   pages={1037--1047},
   issn={0002-9939},
   review={\MR{2045419}},
   doi={10.1090/S0002-9939-03-07300-3},
}

\bib{ADR}{book}{
   author={Anantharaman-Delaroche, C.},
   author={Renault, J.},
   title={Amenable groupoids},
   series={Monographies de L'Enseignement Math\'{e}matique [Monographs of
   L'Enseignement Math\'{e}matique]},
   volume={36},
   note={With a foreword by G. Skandalis and Appendix B by E. Germain},
   publisher={L'Enseignement Math\'{e}matique, Geneva},
   date={2000},
   pages={196},
   isbn={2-940264-01-5},
   review={\MR{1799683}},
}

%\bibitem[BG]{BG} N. Brown and E. Guentner,  {\em New $C^{*}$-completions of discrete groups and related spaces}, Bulletin of the London Mathematical Society, {\bf 45}, 1181-1193,  (2013)

\bib{C}{misc}{
      author={Christensen, Johannes},
       title={The structure of KMS weights on \'etale groupoid C$^*$-algebras},
         how={preprint},
        date={2020},
      eprint={\href{http://arxiv.org/abs/2005.01792}{\texttt{arXiv:2005.01792 [math.OA]}}},
}

\bib{D}{article}{
   author={Deaconu, Valentin},
   title={Groupoids associated with endomorphisms},
   journal={Trans. Amer. Math. Soc.},
   volume={347},
   date={1995},
   number={5},
   pages={1779--1786},
   issn={0002-9947},
   review={\MR{1233967}},
   doi={10.2307/2154972},
}


\bib{ER}{article}{
   author={Exel, R.},
   author={Renault, J.},
   title={Semigroups of local homeomorphisms and interaction groups},
   journal={Ergodic Theory Dynam. Systems},
   volume={27},
   date={2007},
   number={6},
   pages={1737--1771},
   issn={0143-3857},
   review={\MR{2371594}},
   doi={10.1017/S0143385707000193},
}

\bib{HL}{misc}{
      author={Hazrat, Roozbeh},
      author={Li, Huanhuan},
       title={Homology of \'etale groupoids, a graded approach},
         how={preprint},
        date={2018},
      eprint={\href{http://arxiv.org/abs/1806.03398v2}{\texttt{1806.03398v2 [math.KT]}}},
}

\bib{HLS}{article}{
   author={Higson, N.},
   author={Lafforgue, V.},
   author={Skandalis, G.},
   title={Counterexamples to the Baum-Connes conjecture},
   journal={Geom. Funct. Anal.},
   volume={12},
   date={2002},
   number={2},
   pages={330--354},
   issn={1016-443X},
   review={\MR{1911663}},
   doi={10.1007/s00039-002-8249-5},
}

\bib{KhSk}{article}{
   author={Khoshkam, Mahmood},
   author={Skandalis, Georges},
   title={Regular representation of groupoid $C$*-algebras and applications to inverse semigroups},
   journal={J. Reine Angew. Math.},
   volume={546},
   date={2002},
   pages={47--72},
   doi={10.1515/crll.2002.045},
}


\bib{KS}{article}{
   author={Kyed, David},
   author={So\l tan, Piotr M.},
   title={Property (T) and exotic quantum group norms},
   journal={J. Noncommut. Geom.},
   volume={6},
   date={2012},
   number={4},
   pages={773--800},
   issn={1661-6952},
   review={\MR{2990124}},
   doi={10.4171/JNCG/105},
}

\bib{N}{article}{
   author={Neshveyev, Sergey},
   title={KMS states on the $C^\ast$-algebras of non-principal groupoids},
   journal={J. Operator Theory},
   volume={70},
   date={2013},
   number={2},
   pages={513--530},
   issn={0379-4024},
   review={\MR{3138368}},
   doi={10.7900/jot.2011sep20.1915},
}

\bib{NS}{misc}{
      author={Neshveyev, Sergey},
      author={Stammeier, Nicolai},
       title={The groupoid approach to equilibrium states on right LCM semigroup C$^*$-algebras},
         how={preprint},
        date={2019},
      eprint={\href{http://arxiv.org/abs/1912.03141}{\texttt{arXiv:1912.03141 [math.OA]}}},
}

\bib{Rbook}{book}{
   author={Renault, Jean},
   title={A groupoid approach to $C^{\ast} $-algebras},
   series={Lecture Notes in Mathematics},
   volume={793},
   publisher={Springer, Berlin},
   date={1980},
   pages={ii+160},
   isbn={3-540-09977-8},
   review={\MR{584266}},
}

\bib{Ren}{article}{
   author={Renault, Jean},
   title={Repr\'{e}sentation des produits crois\'{e}s d'alg\`ebres de groupo\"{\i}des},
   language={French},
   journal={J. Operator Theory},
   volume={18},
   date={1987},
   number={1},
   pages={67--97},
   issn={0379-4024},
   review={\MR{912813}},
}

\bib{Tak}{book}{
   author={Takesaki, M.},
   title={Theory of operator algebras. I},
   series={Encyclopaedia of Mathematical Sciences},
   volume={124},
   note={Reprint of the first (1979) edition;
   Operator Algebras and Non-commutative Geometry, 5},
   publisher={Springer-Verlag, Berlin},
   date={2002},
   pages={xx+415},
   isbn={3-540-42248-X},
   review={\MR{1873025}},
}

%\bibitem[T]{Thomconf} K. Thomsen, {\em KMS states and conformal measures }, Communications in Mathematical Physics volume  {\bf 316}, (2012)

\bib{W}{article}{
   author={Willett, Rufus},
   title={A non-amenable groupoid whose maximal and reduced $C^*$-algebras
   are the same},
   journal={M\"{u}nster J. Math.},
   volume={8},
   date={2015},
   number={1},
   pages={241--252},
   issn={1867-5778},
   review={\MR{3549528}},
   doi={10.17879/65219671638},
}

\end{biblist}
\end{bibdiv}
\end{document}